\documentclass[11pt,a4paper]{amsart}
\usepackage{amssymb,amscd,amsmath}
\usepackage{pslatex}
\input xy
\xyoption{all}

\makeatletter

\makeatletter
\renewcommand{\phi}{\varphi}
\renewcommand{\geq}{\geqslant}
\renewcommand{\leq}{\leqslant}
\renewcommand{\epsilon}{\varepsilon}

\renewcommand{\kappa}{\varkappa}

\DeclareMathOperator{\spec}{\sf Spec}  \DeclareMathOperator{\thick}{thick}
 \DeclareMathOperator{\inj}{\sf
Sp} \DeclareMathOperator{\Hom}{Hom} \DeclareMathOperator{\supp}{\sf
supp} \DeclareMathOperator{\open}{{open}}
\DeclareMathOperator{\supph}{\sf supph}

\DeclareMathOperator{\End}{End}

 \DeclareMathOperator{\fg}{fg}
\DeclareMathOperator{\fp}{fp} \DeclareMathOperator{\Qcoh}{Qcoh}
\DeclareMathOperator{\coh}{coh} \DeclareMathOperator{\kr}{Ker}
 
\DeclareMathOperator{\Ext}{Ext} \DeclareMathOperator{\im}{Im}

\DeclareMathOperator{\Proj}{\sf Proj}
\DeclareMathOperator{\perf}{\cc D_{per}}
 
\DeclareMathOperator{\Mod}{Mod} \DeclareMathOperator{\modd}{mod}
 \DeclareMathOperator{\loc}{loc}
\DeclareMathOperator{\floc}{f.loc}

\newcommand {\lp}{\varinjlim}
\newcommand {\lato}{L_{\open}}
\newcommand {\latl}{L_{\loc}}
\newcommand {\latfl}{L_{\floc}}
\newcommand {\lattfl}{L_{\floc,\otimes}}
\newcommand{\lo}{\varprojlim}
\newcommand{\lra}[1]{\bl{#1}\longrightarrow\relax}
\newcommand{\bl}[1]{\buildrel #1\over}
\newcommand{\cc}{\mathcal}
\newcommand{\ps}{\oplus}
\newcommand{\ff}{\mathfrak}

\newcommand{\wh}{\widehat}
\newcommand{\wt}{\widetilde}
\newcommand{\ifff}{if and only if }

\newcommand{\lfp}{R \modd}
\newcommand{\bb}{\mathbb}

\newcommand{\Rfp}{\Mod R}

\newtheorem{thm}{Theorem}[section]
\newtheorem{prop}[thm]{Proposition}
\newtheorem{cor}[thm]{Corollary}
\newtheorem{lem}[thm]{Lemma}
\newtheorem*{red}{Reduction principle}
\newtheorem*{theo}{Theorem}

\newtheorem*{rem}{Remark}

\newtheorem*{defs}{Definition}

\numberwithin{equation}{section}

\makeatother

\begin{document}

\footskip30pt

\title{Classifying finite localizations of quasi-coherent sheaves}
\author{Grisha Garkusha}
\address{Department of Mathematics, University of Wales Swansea, Singleton Park, SA2 8PP Swansea, UK}
\email{G.Garkusha@swansea.ac.uk}

\urladdr{http://www-maths.swan.ac.uk/staff/gg}


\begin{abstract}
Given a quasi-compact, quasi-separated scheme $X$, a bijection
between the tensor localizing subcategories of finite type in
$\Qcoh(X)$ and the set of all subsets $Y\subseteq X$ of the form
$Y=\bigcup_{i\in\Omega}Y_i$, with $X\setminus Y_i$ quasi-compact and
open for all $i\in\Omega$, is established. As an application, there
is constructed an isomorphism of ringed spaces
   $$(X,\cc O_X)\lra{\sim}(\spec(\Qcoh(X)),\cc O_{\Qcoh(X)}),$$
where $(\spec(\Qcoh(X)),\cc O_{\Qcoh(X)})$ is a ringed space
associated to the lattice of tensor localizing subcategories of
finite type. Also, a bijective correspondence between the tensor
thick subcategories of perfect complexes $\perf(X)$ and the tensor
localizing subcategories of finite type in $\Qcoh(X)$ is
established.
\end{abstract}


\maketitle

\newdir{ >}{{}*!/-6pt/@{>}} 

\thispagestyle{empty} \pagestyle{plain} \tableofcontents

\section{Introduction}

In his celebrated work on abelian categories P.~Gabriel~\cite{Ga}
proved that for any noetherian scheme $X$ the assignments
   \begin{equation}\label{gabr}
    \coh X\supseteq\cc D\mapsto\bigcup_{x\in\cc D}\supp_X(x)\quad\text{and}\quad X\supseteq
    U\mapsto \{x\in\coh X\mid\supp_X(x)\subseteq U\}
   \end{equation}
induce bijections between
\begin{enumerate}
 \item the set of all tensor Serre subcategories of $\coh X$, and
 \item the set of all subsets $U\subseteq X$ of the form
       $U=\bigcup_{i\in \Omega} Y_i$ where, for all $i \in \Omega$, $Y_i$
       has quasi-compact open complement
       $X\setminus Y_i$.
\end{enumerate}
As a consequence of this result, $X$ can be reconstructed from its
abelian category, $\coh X$, of coherent sheaves  (see
Buan-Krause-Solberg~\cite[Sec.~8]{BKS}). Garkusha and
Prest~\cite{GP,GP1,GP2} have proved similar classification and
reconstruction results for affine and projective schemes.

Given a quasi-compact, quasi-separated scheme $X$, let $\perf(X)$
denote the derived category of perfect complexes. It comes equipped
with a tensor product $\otimes:=\otimes^L_X$. A thick triangulated
subcategory $\cc T$ of $\perf(X)$ is said to be a tensor subcategory
if for every $E \in \perf(X)$ and every object $A\in \cc T$, the
tensor product $E\otimes A$ also is in $\cc T$. Thomason~\cite{T}
establishes a classification similar to~\eqref{gabr} for tensor
thick subcategories of $\perf(X)$ in terms of the topology of $X$.
Hopkins and Neeman (see~\cite{Hop,N}) did the case where $X$ is
affine and noetherian.

Based on Thomason's classification theorem, Balmer~\cite{B1}
reconstructs the noetherian scheme $X$ from the tensor thick
triangulated subcategories of $\perf(X)$. This result has been
generalized to quasi-compact, quasi-separated schemes by
Buan-Krause-Solberg~\cite{BKS}.

The main result of this paper is a generalization of the
classification result by Garkusha and Prest~\cite{GP,GP1,GP2} to
schemes. Let $X$ be a quasi-compact, quasi-separated scheme. Denote
by $\Qcoh(X)$ the category of quasi-coherent sheaves. We say that a
localizing subcategory $\cc S$ of $\Qcoh(X)$ is of finite type if
the canonical functor from the quotient category $\Qcoh(X)/\cc
S\to\Qcoh(X)$ preserves directed sums.

\begin{theo}[Classification]
Let $X$ be a quasi-compact, quasi-separated scheme. Then the maps
   $$V\mapsto\mathcal S=\{\cc F\in\Qcoh(X)\mid\supp_X(\cc F)\subseteq V\}$$
and
   $$\mathcal S\mapsto V=\bigcup_{\cc F\in\mathcal S}\supp_X(\cc F)$$ induce
bijections between
\begin{enumerate}
 \item the set of all subsets of
the form $V=\bigcup_{i\in\Omega}V_i$ with quasi-compact open
complement $X\setminus V_i$ for all $i\in\Omega$,

 \item the set of all tensor localizing subcategories of finite type in $\Qcoh(X)$.
\end{enumerate}
\end{theo}

As an application of the Classification Theorem, we show that there
is a 1-1 correspondence between the tensor finite localizations in
$\Qcoh(X)$ and the tensor thick subcategories in $\perf(X)$
(cf.~\cite{Ho,GP,GP2}).

\begin{theo}
Let $X$ be a quasi-compact and quasi-separated scheme. The
assignments
   $$\cc T\mapsto\cc S=\{\cc F\in\Qcoh(X)\mid\supp_X(\cc F)
     \subseteq\bigcup_{n\in\bb Z,E\in\cc T}\supp_X(H_n(E))\}$$
and
   $$\cc S\mapsto
     \{E\in\perf(X)\mid H_n(E)\in\cc S\textrm{ for all $n\in\bb Z$}\}$$
induce a bijection between
\begin{enumerate}
\item the set of all tensor thick subcategories of $\perf(X)$,

\item the set of all tensor localizing subcategories of finite type
in $\Qcoh(X)$.
\end{enumerate}
\end{theo}

Another application of the Classification Theorem is the
Recostruction Theorem. A common approach in non-commutative geometry
is to study abelian or triangulated categories and to think of them
as the replacement of an underlying scheme. This idea goes back to
work of Grothendieck and Manin. The approach is justified by the
fact that a noetherian scheme can be reconstructed from the abelian
category of coherent sheaves (Gabriel~\cite{Ga}) or from the
category of perfect complexes (Balmer~\cite{B1}). Rosenberg~\cite{R}
proved that a quasi-compact scheme $X$ is reconstructed from its
category of quasi-coherent sheaves.

In this paper we reconstruct a quasi-compact, quasi-separated scheme
$X$ from $\Qcoh(X)$. Our approach, similar to that used
in~\cite{GP,GP1,GP2}, is entirely different from
Rosenberg's~\cite{R} and less abstract.

Following Buan-Krause-Solberg~\cite{BKS} we consider the lattice
$\lattfl(X)$ of tensor localizing subcategories of finite type in
$\Qcoh(X)$ as well as its prime ideal spectrum $\spec(\Qcoh(X))$.
The space comes naturally equipped with a sheaf of commutative rings
$\cc O_{\Qcoh(X)}$. The following result says that the scheme
$(X,\cc O_X)$ is isomorphic to $(\spec(\Qcoh(X)),\cc O_{\Qcoh(X)})$.

\begin{theo}[Reconstruction]
Let $X$ be a quasi-compact and quasi-sepa\-rated scheme. Then there
is a natural isomorphism of ringed spaces
   $$f:(X,\cc O_X)\lra{\sim}(\spec(\Qcoh(X)),\cc O_{\Qcoh(X)}).$$
\end{theo}

Other results presented here worth mentioning are the theorem
classifying finite localizations in a locally finitely presented
Grothendieck category $\cc C$ (Theorem~\ref{zg}) in terms of some
topology on the injective spectrum $\inj\cc C$, generalizing a
result of Herzog~\cite{H} and Krause~\cite{Kr1} for locally coherent
Grothendieck categories, and the Classification and Reconstruction
Theorems for coherent schemes.

\section{Localization in Grothendieck categories}

The category $\Qcoh(X)$ of quasi-coherent sheaves over a scheme $X$
is a Gro\-then\-dieck category (see~\cite{E}), so hence we can apply
the general localization theory for Grothendieck categories which is
of great utility in our analysis. For the convenience of the reader
we shall recall some basic facts of this theory.

We say that a subcategory $\cc S$ of an abelian category $\cc C$ is
a {\it Serre subcategory\/} if for any short exact sequence
   $$0\to X\to Y\to Z\to 0$$
in $\cc C$ an object $Y\in\cc S$ if and only if $X$, $Z\in\cc S$. A
Serre subcategory $\cc S$ of a Grothendieck category $\cc C$ is {\it
localizing\/} if it is closed under taking direct limits.
Equivalently, the inclusion functor $i:\cc S\to\cc C$ admits the
right adjoint $t=t_{\cc S}:\cc C\to\cc S$ which takes every object
$X\in\cc C$ to the maximal subobject $t(X)$ of $X$ belonging to $\cc
S$. The functor $t$ we call the {\it torsion functor}. An object $C$
of $\cc C$ is said to be {\it $\cc S$-torsionfree\/} if $t(C)=0$.
Given a localizing subcategory $\cc S$ of $\cc C$ the {\it quotient
category $\cc C/\cc S$\/} consists of $C\in\cc C$ such that
$t(C)=t^1(C)=0$. The objects from $\cc C/\cc S$ we call {\it $\cc
S$-closed objects}. Given $C\in\cc C$ there exists a canonical exact
sequence
   $$0\to A'\to C\lra{\lambda_C}C_{\cc S}\to A''\to 0$$
with $A'=t(C)$, $A''\in\cc S$, and where $C_{\cc S}\in\cc C/\cc S$
is the maximal essential extension of $\wt C=C/t(C)$ such that
$C_{\cc S}/\wt C\in\cc S$. The object $C_{\cc S}$ is uniquely
defined up to a canonical isomorphism and is called the {\it $\cc
S$-envelope\/} of $C$. Moreover, the inclusion functor $i:\cc C/\cc
S\to\cc C$ has the left adjoint {\it localizing functor\/} $(-)_{\cc
S}:\cc C\to\cc C/\cc S$, which is also exact. It takes each $C\in\cc
C$ to $C_{\cc S}\in\cc C/\cc S$. Then,
   $$\Hom_{\cc C}(X,Y)\cong\Hom_{\cc C/\cc S}(X_{\cc S},Y)$$
for all $X\in\cc C$ and $Y\in\cc C/\cc S$.

If $\cc C$ and $\cc D$ are Grothendieck categories, $q:\cc C\to\cc
D$ is an exact functor, and a functor $s:\cc D\to\cc C$ is fully
faithful and right adjoint to $q$, then $\cc S:=\kr q$ is a
localizing subcategory and there exists an equivalence $\cc C/\cc
S\bl H\cong\cc D$ such that $H\circ(-)_{\cc S}=q$. We shall refer to
the pair $(q,s)$ as the {\it localization pair}.

The following result is an example of the localization pair.

\begin{prop} (cf.~\cite[\S III.5; Prop. VI.3]{Ga}) \label{loc}
Let $X$ be a scheme, let $U$ be an open subset of $X$ such that the
canonical injection $j:U\to X$ is a quasi-compact map. Then $j_*(\cc
G)$ is a quasi-coherent $\cc O_X$-module for any quasi-coherent $\cc
O_X|_U$-module $\cc G$ and the pair of adjoint functors $(j^*,j_*)$
is a localization pair. That is the category of quasi-coherent $\cc
O_X|_U$-modules $\Qcoh(U)$ is equivalent to $\Qcoh(X)/\cc S$, where
$\cc S=\kr j^*$. Moreover, a quasi-coherent $\cc O_X$-module $\cc F$
belongs to the localizing subcategory $\cc S$ if and only if
$\supp_X(\cc F)=\{P\in X\mid\cc F_P\neq 0\}\subseteq Z=X\setminus
U$. Also, for any $\cc F\in\Qcoh(X)$ we have $t_{\cc S}(\cc F)=\cc
H_Z^0(\cc F)$, where $\cc H_Z^0(\cc F)$ stands for the subsheaf of
$\cc F$ with supports in $Z$.
\end{prop}

\begin{proof}
The fact that $j_*(\cc G)$ is a quasi-coherent $\cc O_X$-module
follows from~\cite[I.6.9.2]{GD}. The functor $j^*:\cc F\mapsto\cc
F|_U$ is clearly exact, $j_*(\cc G)|_U=j^*j_*(\cc G)=\cc G$
by~\cite[I.6.9.2]{GD}. It follows that $j_*$ is fully faithful, and
hence $(j^*,j_*)$ is a localization pair.

The fact that $\cc F\in\cc S$ if and only if $\supp_X(\cc
F)\subseteq Z$ is obvious. Finally, by~\cite[Ex. II.1.20]{Har} we
have an exact sequence
   $$0\to\cc H_Z^0(\cc F)\to\cc F\lra{\rho_{\cc F}}j_*j^*(\cc F).$$
Since the morphism $\rho_{\cc F}$ can be regarded as an $\cc
S$-envelope for $\cc F$, we see that $\kr\rho_{\cc F}=t_{\cc S}(\cc
F)=\cc H_Z^0(\cc F)$.
\end{proof}

Given a subcategory $\cc X$ of a Grothendieck category $\cc C$, we
denote by $\surd\cc X$ the smallest localizing subcategory of $\cc
C$ containing $\cc X$. To describe $\surd\cc X$ intrinsically, we
need the notion of a subquotient.

\begin{defs}{\rm

Given objects $A,B\in\cc C$, we say that $A$ is a {\it
subquotient\/} of $B$, or $A\prec B$, if there is a filtration of
$B$ by subobjects $B=B_0\geq B_1\geq B_2\geq 0$ such that $A\cong
B_1/B_2$. In other words, $A$ is isomorphic to a subobject of a
quotient object of $B$.

}\end{defs}

Given a subcategory $\cc X$ of $\cc C$, we denote by $\langle\cc
X\rangle$ the full subcategory of subquotients of objects from $\cc
X$. Clearly, $\langle\cc X\rangle=\langle\langle\cc
X\rangle\rangle$, for the relation $A\prec B$ is transitive, and
$\cc X=\langle\cc X\rangle$ if and only if $\cc X$ is closed under
subobjects and quotient objects. If $\cc X$ is closed under direct
sums then so is $\langle\cc X\rangle$.

\begin{prop}\label{surd}
Given  a subcategory $\cc X$ of a Grothendieck category $\cc C$, an
object $X\in\surd\cc X$ if and only if there is a filtration
   $$X_0\subset X_1\subset\cdots\subset X_\beta\subset\cdots$$
such that $X=\bigcup_\beta X_\beta$,
$X_\gamma=\bigcup_{\beta<\gamma} X_\beta$ if $\gamma$ is a limit
ordinal, and $X_0,X_{\beta+1}/X_\beta\in\langle\cc X^\ps\rangle$
with $\cc X^\ps$ standing for the subcategory of $\cc C$ consisting
of direct sums of objects in $\cc X$.
\end{prop}

\begin{proof}
It is easy to see that every object having such a filtration belongs
to $\surd\cc X$. It is enough to show that the full subcategory $\cc
S$ of such objects is localizing. Let
   $$X\rightarrowtail Y\bl g\twoheadrightarrow Z$$
be a short exact sequence with $X,Z\in\cc S$. Let
$X_0\subset\cdots\subset X_\beta\subset\cdots$ and
$Z_0\subset\cdots\subset Z_\alpha\subset\cdots$ be the corresponding
filtrations. Put $Y_\alpha=g^{-1}(Z_\alpha)$. Then we have a short
exact sequence for any $\alpha$
   $$X\rightarrowtail Y_\alpha\bl{g_\alpha}\twoheadrightarrow Z_\alpha$$
with $Y_{\alpha+1}/Y_\alpha\cong Z_{\alpha+1}/Z_\alpha$. We have the
following filtration for $Y$:
   $$X_0\subset\cdots\subset X_\beta\subset\cdots\subset X=\bigcup_\beta X_\beta\subset Y_0\subset\cdots\subset Y_\alpha\subset\cdots$$
It follows that $Y\in\cc S$. We see that $\cc S$ is closed under
extensions.

Now let $Y\in\cc S$ with a filtration $Y_0\subset\cdots\subset
Y_\alpha\subset\cdots$. Set $X_\alpha=X\cap Y_\alpha$ and
$Z_\alpha=Y_\alpha/X_\alpha$. We get filtrations
$X_0\subset\cdots\subset X_\alpha\subset\cdots$ and
$Z_0\subset\cdots\subset Z_\alpha\subset\cdots$ for $X$ and $Z$
respectively. Thus $X,Z\in\cc S$, and so $\cc S$ is a Serre
subcategory. It is plainly closed under direct sums, hence it is
localizing.
\end{proof}

\begin{cor}\label{vrp}
Let $\cc X$ be a subcategory in $\cc C$ closed under subobjects,
quotient objects, and direct sums. An object $M\in\cc C$ is
$\surd\cc X$-closed if and only if $\Hom(X,M)=\Ext^1(X,M)=0$ for all
$X\in\cc X$.
\end{cor}

\begin{proof}
Suppose $\Hom(X,M)=\Ext^1(X,M)=0$ for all $X\in\cc X$. We have to
check that $\Hom(Y,M)=\Ext^1(Y,M)=0$ for all $Y\in\surd\cc X$. By
Proposition~\ref{surd} there is a filtration
   $$Y_0\subset Y_1\subset\cdots\subset Y_\beta\subset\cdots$$
such that $Y=\bigcup_\beta Y_\beta$,
$Y_\gamma=\bigcup_{\beta<\gamma}Y_\beta$ if $\gamma$ is a limit
ordinal, and $Y_0,Y_{\beta+1}/Y_\beta\in\langle\cc X^\ps\rangle=\cc
X$. One has an exact sequence for any $\beta$
   \begin{gather*}
    \Hom(Y_{\beta+1}/Y_\beta,M)\to\Hom(Y_{\beta+1},M)\to\Hom(Y_\beta,M)\to\\
    \to\Ext^1(Y_{\beta+1}/Y_\beta,M)\to\Ext^1(Y_{\beta+1},M)\to\Ext^1(Y_\beta,M).
   \end{gather*}
One sees that if $\Hom(Y_\beta,M)=\Ext^1(Y_{\beta},M)=0$ then
$\Hom(Y_{\beta+1},M)=\Ext^1(Y_{\beta+1},M)=0$, because
$Y_{\beta+1}/Y_\beta\in\cc X$. Since $Y_0\in\cc X$ it follows that
$\Hom(Y_{\beta},M)=\Ext^1(Y_{\beta},M)=0$ for all finite $\beta$.

Let $\gamma$ be a limit ordinal and
$\Hom(Y_{\beta},M)=\Ext^1(Y_{\beta},M)=0$ for all $\beta<\gamma$. We
have $\Hom(Y_{\gamma},M)=\lo_{\beta<\gamma}\Hom(Y_{\beta},M)=0$. Let
us show that $\Ext^1(Y_{\gamma},M)=0$. To see this we must prove
that every short exact sequence
   $$M\rightarrowtail N\twoheadrightarrow Y_\gamma$$
is split. One can construct a commutative diagram
   $$\xymatrix{
    E_\beta:& M\ar@{=}[d]\ar@{ >->}[r]&N_\beta\ar@{ >->}[d]\ar@{->>}[r]^{p_\beta}&Y_{\beta}\ar@{ >->}[d]\ar@/^/[l]^{\exists\kappa_\beta}\\
    E_\gamma:& M\ar@{ >->}[r]&N\ar@{->>}[r]^(.45)p&Y_\gamma
     }$$
with $N_\beta=p^{-1}(Y_\beta)$. Clearly, $E_\gamma=\bigcup_\beta
E_\beta$. Since the upper row splits, there exists a morphism
$\kappa_\beta$ such that $p_\beta\kappa_\beta=1$. Consider the
following commutative diagram:
   $$\xymatrix{
     M\ar@{=}[d]\ar@{ >->}[r]&N_\beta\ar@{ >->}[d]_{u_\beta}\ar@{->>}[r]^{p_\beta}&Y_{\beta}\ar@{ >->}[d]^{v_\beta}\\
     M\ar@{
     >->}[r]&N_{\beta+1}\ar@{->>}[r]^(.45){p_{\beta+1}}&Y_{\beta+1}
     }$$
We want to check that
$\kappa_{\beta+1}v_{\beta}=u_\beta\kappa_{\beta}$. Since the right
square is cartesian and
$p_{\beta+1}\kappa_{\beta+1}v_{\beta}=v_\beta$, there exists a
unique morphism $\tau:Y_\beta\to N_\beta$ such that
$p_\beta\tau=1$ and $u_\beta\tau=\kappa_{\beta+1}v_{\beta}$. We
claim that $\tau=\kappa_\beta$. Indeed,
$p_\beta(\tau-\kappa_\beta)=0$ and hence $\tau-\kappa_\beta$
factors through $M$. The latter is possible only if
$\tau-\kappa_\beta=0$, for $\Hom(Y_\beta,M)=0$ by assumption.
Therefore $\tau=\kappa_\beta$. It follows that the family of
morphisms $\kappa_\beta:Y_\beta\to N_\beta$ is directed and then
$p\circ(\lp\kappa_\beta)=(\lp p_\beta)\circ(\lp\kappa_\beta)=\lp
(p_\beta\kappa_\beta)=1$. Thus $p$ is split.
\end{proof}

Recall that the {\it injective spectrum\/} or the {\it Gabriel
spectrum\/} $\inj\cc C$ of a Grothendieck category $\cc C$ is the
set of isomorphism classes of injective objects in $\cc C$. It
plays an important role in our analysis. Given a subcategory $\cc
X$ in $\cc C$ we denote by
   $$(\cc X)=\{E\in\inj\cc C\mid\Hom_{\cc C}(X,E)\neq 0\textrm{ for some } X\in\cc X\}.$$
Using Proposition~\ref{loc} and the fact that the functor
$\Hom(-,E)$, $E\in\inj\cc C$, is exact, we have $(\cc
X)=\bigcup_{X\in\cc X}(X)=(\surd\cc X)$.

\begin{prop}\label{tp}
The collection of subsets of $\inj\cc C$,
   $$\{(\cc S)\mid\cc S\subset\cc C \textrm{ is a localizing subcategory}\},$$
satisfies the axioms for the open sets of a topology on the
injective spectrum $\inj\cc C$. This topological space will be
denoted by $\inj_{gab}\cc C$. Moreover, the map
   \begin{equation}\label{11}
    \cc S\longmapsto(\cc S)
   \end{equation}
is an inclusion-preserving bijection between the localizing
subcategories $\cc S$ of $\cc C$ and the open subsets of
$\inj_{gab}\cc C$.
\end{prop}

\begin{proof}
First note that $(0)=\emptyset$ and $(\cc C)=\inj\cc C$. We have
$(\cc S_1)\cap(\cc S_2)=(\cc S_1\cap\cc S_2)$ because every
$E\in\inj\cc C$ is uniform and $0\neq t_{\cc S_1}(E)\cap t_{\cc
S_2}(E)\in\cc S_1\cap\cc S_2$ whenever $E\in(\cc S_1)\cap(\cc
S_2)$. Also, $\cup_{i\in I}(\cc S_i)=(\cup_{i\in I}\cc
S_i)=(\surd\cup_{i\in I}\cc S_i)$.

The map~\eqref{11} is plainly bijective, because every localizing
subcategory $\cc S$ consists precisely of those objects $X$ such
that $\Hom(X,E)=0$ for all $E\in\inj\cc C\setminus(\cc S)$.
\end{proof}

Given a localizing subcategory $\cc S$ in $\cc C$, the injective
spectrum $\inj_{gab}(\cc C/\cc S)$ can be considered as the closed
subset $\inj_{gab}\cc C\setminus(\cc S)$. Moreover, the inclusion
   $$\inj_{gab}(\cc C/\cc S)\hookrightarrow\inj_{gab}\cc C$$
is a closed map. Indeed, if $U$ is a closed subset in
$\inj_{gab}(\cc C/\cc S)$ then there is a unique localizing
subcategory $\cc T$ in $\cc C/\cc S$ such that $U=\inj_{gab}(\cc
C/\cc S)\setminus(\cc T)$. By~\cite[1.7]{G} there is a unique
localizing subcategory $\cc P$ in $\cc C$ containing $\cc S$ such
that $\cc C/\cc P$ is equivalent to $(\cc C/\cc S)/\cc T$. It
follows that $U=\inj_{gab}\cc C\setminus(\cc P)$, hence $U$ is
closed in $\inj_{gab}\cc C$.

On the other hand, let $\cc Q$ be a localizing subcategory of $\cc
C$. Let us show that $O:=(\cc Q)\cap\inj_{gab}(\cc C/\cc S)$ is an
open subset in $\inj_{gab}(\cc C/\cc S)$.

\begin{lem}\label{lenya}
Let $\wh{\cc Q}$ denote the full subcategory of objects of the
form $X_{\cc S}$ with $X\in\cc Q$. Then $\wh{\cc Q}$ is closed
under direct sums, subobjects, quotient objects in $\cc C/\cc S$
and $O=(\surd\wh{\cc Q})$. Moreover, if $\cc T$ is the unique
localizing subcategory of $\cc C$ containing $\cc S$ such that
$\cc C/\cc T\cong(\cc C/\cc S)/\surd\wh{\cc Q}$ then the following
relation is true:
   $$\cc T=\surd({\cc Q}\cup\cc S),$$
that is $\cc T$ is the smallest localizing subcategory containing
$\cc Q$ and $\cc S$. We shall also refer to $\cc T$ as the {\em
join} of $\cc Q$ and $\cc S$.
\end{lem}

\begin{proof}
First let us prove that $\wh{\cc Q}$ is closed under direct sums,
subobjects, quotient objects in $\cc C/\cc S$. It is plainly
closed under direct sums. Let $Y$ be a subobject of $X_{\cc S}$,
$X\in\cc Q$, and let $\lambda_X:X\to X_{\cc S}$ be the $\cc
S$-envelope for $X$. Then $W=\lambda^{-1}_X(Y)$ is a subobject of
$X$, hence it belongs to $\cc Q$, and $Y=W_{\cc S}$. If $Z$ is a
$\cc C/\cc S$-quotient of $X_{\cc S}$ and $\pi:X_{\cc
S}\twoheadrightarrow Z$ is the canonical projection, then
$Z=V_{\cc S}$ with $V=X/{\kr({\pi\lambda_X}})\in\cc Q$. So
$\wh{\cc Q}$ is also closed under subobjects and quotient objects
in $\cc C/\cc S$.

It follows that $\langle\wh{\cc Q}^\ps\rangle=\wh{\cc Q}$ and
$(\wh{\cc Q})=(\surd\wh{\cc Q})$. On the other hand, $O=(\wh{\cc
Q})$ as one easily sees. Thus $O$ is open in $\inj_{gab}(\cc C/\cc
S)$.

Clearly,
   $$(\cc T)=(\cc Q)\cup(\cc S)=(\cc Q\cup\cc S)=(\surd(\cc Q\cup\cc S)).$$
By Proposition~\ref{tp} $\cc T=\surd(\cc Q\cup\cc S)$.
\end{proof}

We summarize the above arguments as follows.

\begin{prop}\label{len}
Given a localizing subcategory $\cc S$ in $\cc C$, the topology on
$\inj_{gab}(\cc C/\cc S)$ coincides with the subspace topology
induced by $\inj_{gab}\cc C$.
\end{prop}

\section{Finite localizations of Grothendieck categories}

In this paper we are mostly interested in finite localizations of
a Grothendieck category $\cc C$. For this we should assume some
finiteness conditions for $\cc C$.

Recall that an object $X$ of a Grothendieck category $\cc C$ is
{\it finitely generated\/} if whenever there are subobjects
$X_i\subseteq X$ with $i\in I$ satisfying $X=\sum_{i\in I}X_i$,
then there is a finite subset $J\subset I$ such that $X=\sum_{i\in
J}X_i$. The subcategory of finitely generated objects is denoted
by $\fg\cc C$. A finitely generated object $X$ is said to be {\it
finitely presented\/} if every epimorphism $\gamma:Y\to X$ with
$Y\in\fg\cc C$ has the finitely generated kernel $\kr\gamma$. By
$\fp\cc C$ we denote the subcategory consisting of finitely
presented objects. The category $\cc C$ is {\em locally finitely
presented\/} if every object $C\in\cc C$ is a direct limit $C=\lp
C_i$ of finitely presented objects $C_i$, or equivalently, $\cc C$
possesses a family of finitely presented generators. In such a
category, every finitely generated object $A\in\cc C$ admits an
epimorphism $\eta:B\to A$ from a finitely presented object $B$.
Finally, we refer to a finitely presented object $X\in\cc C$ as
{\it coherent\/} if every finitely generated subobject of $X$ is
finitely presented. The corresponding subcategory of coherent
objects will be denoted by $\coh\cc C$. A locally finitely
presented category $\cc C$ is {\em locally coherent\/} if $\coh\cc
C=\fp\cc C$. Obviously, a locally finitely presented category $\cc
C$ is locally coherent if and only if $\coh\cc C$ is an abelian
category.

In~\cite{E} it is shown that the category of quasi-coherent sheaves
$\Qcoh(X)$ over a scheme $X$ is a locally $\lambda$-presentable
category, for $\lambda$ a certain regular cardinal. However for some
nice schemes which are in practise the most used for algebraic
geometers like quasi-compact and quasi-separated there are enough
finitely presented generators for $\Qcoh(X)$.

\begin{prop}\label{qcoh}
Let $X$ be a quasi-compact and quasi-separated scheme. Then
$\Qcoh(X)$ is a locally finitely presented Grothendieck category. An
object $\cc F\in\fp(\Qcoh(X))$ \ifff it is locally finitely
presented.
\end{prop}

\begin{proof}
By~\cite[I.6.9.12]{GD} every quasi-coherent sheaf is a direct limit
of locally finitely presented sheaves. It follows from~\cite[Prop.
75]{M} that the locally finitely presented sheaves are precisely the
finitely presented objects in $\Qcoh(X)$.
\end{proof}

Recall that a localizing subcategory $\cc S$ of a Grothendieck
category $\cc C$ is {\it of finite type\/} (respectively {\it of
strictly finite type}) if the functor $i:\cc C/\cc S\to\cc C$
preserves directed sums (respectively direct limits). If $\cc C$
is a locally finitely generated (respectively, locally finitely
presented) Grothen\-dieck category and $\cc S$ is of finite type
(respectively, of strictly finite type), then $\cc C/\cc S$ is a
locally finitely generated (respectively, locally finitely
presented) Grothendieck category and
   $$\fg(\cc C/\cc S)=\{C_{\cc S}\mid C\in\fg\cc C\}\quad\textrm{(respectively
     $\fp(\cc C/\cc S)=\{C_{\cc S}\mid C\in\fp\cc C\}$)}.$$
If $\cc C$ is a locally coherent Grothendieck category then $\cc
S$ is of finite type \ifff it is of strictly finite type (see,
e.g.,~\cite[5.14]{G}). In this case $\cc C/\cc S$ is locally
coherent.

The following proposition says that localizing subcategories of
finite type in a locally finitely presented Grothendieck category
$\cc C$ are completely determined by finitely presented torsion
objects (cf.~\cite{H,Kr1}).

\begin{prop}\label{zu}
Let $\cc S$ be a localizing subcategory of finite type in a locally
finitely presented Grothendieck category $\cc C$. Then the following
relation is true:
   $$\cc S=\surd(\fp\cc C\cap\cc S).$$
\end{prop}

\begin{proof}
Obviously, $\surd(\fp\cc C\cap\cc S)\subset\cc S$. Let $X\in\cc S$
and let $Y$ be a finitely generated subobject of $X$. There is an
epimorphism $\eta:Z\to Y$ with $Z\in\fp\cc C$. By~\cite[5.8]{G}
there is a finitely generated subobject $W\subset\kr\eta$ such that
$Z/W\in\cc S$. It follows that $Z/W\in\fp\cc C\cap\cc S$ and $Y$ is
an epimorphic image of $Z/W$. Since $X$ is a direct union of
finitely generated torsion subobjects, we see that $X$ is an
epimorphic image of some $\ps_{i\in I}S_i$ with each $S_i\in\fp\cc
C\cap\cc S$. Therefore $\cc S\subset\surd(\fp\cc C\cap\cc S)$.
\end{proof}

\begin{lem}\label{gg}
Let $\cc Q$ and $\cc S$ be two localizing subcategories in a
Grothendieck category $\cc C$. If $X\in\cc C$ is both $\cc
Q$-closed and $\cc S$-closed, then it is $\cc T=\surd(\cc Q\cup\cc
S)$-closed.
\end{lem}

\begin{proof}
By Lemma~\ref{lenya} $\cc C/\cc T\cong(\cc C/\cc S)/\surd\wh{\cc
Q}$, where $\wh{\cc Q}=\{C_{\cc S}\in\cc C/\cc S\mid C\in\cc Q\}$
and $\wh{\cc Q}$ is closed under direct sums, subobjects, quotient
objects in $\cc C/\cc S$. To show that $X=X_{\cc S}$ is a $\cc
T$-closed object it is enough to check that $X$ is $\surd\wh{\cc
Q}$-closed in $\cc C/\cc S$. Obviously $\Hom_{\cc C/\cc S}(A,X)=0$
for all $A\in\wh{\cc Q}$.

Consider a short exact sequence in $\cc C/\cc S$
   $$E:\quad X\rightarrowtail Y\bl p\twoheadrightarrow C_{\cc S}$$
with $C\in{\cc Q}$. One can construct a commutative diagram in
$\cc C$
   $$\xymatrix{
    E':& X\ar@{=}[d]\ar@{ >->}[r]&Y'\ar[d]\ar[r]^{p'}&C\ar[d]^{\lambda_C}\\
    E:& X\ar@{ >->}[r]&Y\ar[r]^(.45)p&C_{\cc S}
     }$$
with the right square cartesian. Let $C'=\im_{\cc C}p'$ then
$C'\in\cc Q$, $C'_{\cc S}=C_{\cc S}$ and the short exact sequence
   $$E'':\quad X\rightarrowtail Y'\bl{p'}\twoheadrightarrow C'$$
splits, because $\Ext^1_{\cc C}(C',X)=0$. It follows that $E$
splits for $E=E'_{\cc S}=E''_{\cc S}$. Therefore $X$ is
$\surd\wh{\cc Q}$-closed by Corollary~\ref{vrp}.
\end{proof}

Below we shall need the following

\begin{lem}\label{ggg}
Given a family of localizing subcategories of finite type $\{\cc
S_i\}_{i\in I}$ in a locally finitely presented Grothendieck
category $\cc C$, their join $\cc T=\surd(\cup_{i\in I}\cc S_i)$
is a localizing subcategory of finite type.
\end{lem}

\begin{proof}
Let us first consider the case when $I$ is finite. By induction it
is enough to show that the join $\cc T=\surd(\cc Q\cup\cc S)$ of
two localizing subcategories of finite type $\cc Q$ and $\cc S$ is
of finite type. We have to check that the inclusion functor $\cc
C/\cc T\to\cc C$ respects directed sums. It is plainly enough to
verify that $X=\sum_{\cc C}X_\alpha$ is a $\cc T$-closed object
whenever each $X_\alpha$ is $\cc T$-closed. Since $\cc Q$ and $\cc
S$ are of finite type, $X$ is both $\cc Q$-closed and $\cc
S$-closed. It follows from Lemma~\ref{gg} that $X$ is $\cc
T$-closed. Therefore $\cc T$ is of finite type.

Now let $\{\cc S_i\}_{i\in I}$ be an arbitrary set of localizing
subcategories of finite type. Without loss of generality we may
assume that $I$ is a directed set and $\cc S_i\subset\cc S_j$ for
$i\le j$. Indeed, given a finite subset $J\subset I$ we denote by
$\cc S_J$ the localizing subcategory of finite type
$\surd(\cup_{j\in J}\cc S_j)$. Then the set $R$ of all finite
subsets $J$ of $I$ is plainly directed, $\cc S_J\subset\cc S_{J'}$
for any $J\subset J'$, and $\cc T=\surd(\cup_{J\in R}\cc S_J)$.

Let $\cc X$ denote the full subcategory of $\cc C$ of those objects
which can be presented as directed sums $\sum X_\alpha$ with each
$X_\alpha$ belonging to $\cup_{i\in I}\cc S_i$. Since $I$ is a
directed set and $\cc S_i\subset\cc S_j$ for $i\le j$, it follows
that a direct sum $X=\ps_{\gamma\in\Gamma}X_\gamma$ with each
$X_\gamma$ belonging to $\cup_{i\in I}\cc S_i$. is in $\cc X$.
Indeed, $X=\sum X_S$ with $S$ running through all finite subsets of
$\Gamma$ and $X_S=\ps_{\gamma\in S}X_\gamma\in\cup_{i\in I}\cc S_i$.
Therefore if $\{X_\beta\}_{\beta\in B}$ is a family of subobjects of
an object $X$ and each $X_\beta$ belongs to $\cup_{i\in I}\cc S_i$,
then the direct union $\sum X_\beta$ belongs to $\cc X$.

The subcategory $\cc X$ is closed under subobjects and quotient
objects. Indeed, let $X=\sum X_\alpha$ with each $X_\alpha$
belonging to $\cup_{i\in I}\cc S_i$. Consider a short exact
sequence
   $$Y\rightarrowtail X\twoheadrightarrow Z.$$
We set $Y_\alpha=Y\cap X_\alpha$ and
$Z_\alpha=X_\alpha/Y_\alpha\subset Z$. Then both $Y_\alpha$ and
$Z_\alpha$ are in $\cup_{i\in I}\cc S_i$, $Y=Y\cap(\sum
X_\alpha)=\sum Y\cap X_\alpha=\sum Y_\alpha$ and $Z=\sum
Z_\alpha$. So $Y,Z\in\cc X$.

Clearly, $\cc X$ is closed under directed sums, in particular
direct sums, hence $\cc X=\langle\cc X^\ps\rangle$ and $\cc
T=\surd\cc X$. If we show that every direct limit $C=\lp C_\delta$
of $\cc T$-closed objects $C_\delta$ has no $\cc T$-torsion, it
will follow from~\cite[5.8]{G} that $\cc T$ is of finite type.

Using Proposition~\ref{surd}, it is enough to check that
$\Hom_{\cc C}(X,C)=0$ for any object $X\in\cc X$. Let $Y$ be a
finitely generated subobject in $X$. There is an index $i_0\in I$
such that $Y\in\cc S_{i_0}$ and an epimorphism
$\eta:Z\twoheadrightarrow Y$ with $Z\in\fp\cc C$. By~\cite[5.8]{G}
there exists a finitely generated subobject $W$ of $\kr\eta$ such
that $Z/W\in\cc S_{i_0}$. Since $Z/W\in\fp\cc C$ then
$\Hom(Z/W,C)=\lp\Hom(Z/W,C_\delta)=0$. We see that $\Hom(Y,C)=0$,
and hence $\Hom(X,C)=0$.
\end{proof}

Given a localizing subcategory of finite type $\cc S$ in $\cc C$, we
denote by
   $$O(\cc S)=\{E\in\inj\cc C\mid t_{\cc S}(E)\neq 0\}.$$

The next result has been obtained by Herzog~\cite{H} and
Krause~\cite{Kr1} for locally coherent Grothendieck categories and
by Garkusha-Prest~\cite{GP2} for the category of modules $\Rfp$ over
a commutative ring $R$.

\begin{thm}\label{zg}
Suppose $\cc C$ is a locally finitely presented Grothendieck
category. The collection of subsets of $\inj\cc C$,
   $$\{O(\cc S)\mid\cc S\subset\cc C \textrm{ is a localizing subcategory of finite type}\},$$
satisfies the axioms for the open sets of a topology on the
injective spectrum $\inj\cc C$. This topological space will be
denoted by $\inj_{fl}\cc C$ and this topology will be referred to as
the {\em fl-topology} (``fl" for finite localizations). Moreover,
the map
   \begin{equation}\label{111}
    \cc S\longmapsto O(\cc S)
   \end{equation}
is an inclusion-preserving bijection between the localizing
subcategories $\cc S$ of finite type in $\cc C$ and the open subsets
of $\inj_{fl}\cc C$.
\end{thm}

\begin{proof}
First note that $O(\cc S)=(\cc S)$, $O(0)=\emptyset$ and $O(\cc
C)=\inj\cc C$. We have $O(\cc S_1)\cap O(\cc S_2)=(\cc S_1\cap\cc
S_2)$ by Proposition~\ref{tp}. We claim that $\cc S_1\cap\cc S_2$ is
of finite type, whence $O(\cc S_1)\cap O(\cc S_2)=O(\cc S_1\cap\cc
S_2)$. Indeed, let us consider a morphism $f:X\to S$ from a finitely
presented object $X$ to an object $S\in\cc S_1\cap\cc S_2$. It
follows from~\cite[5.8]{G} that there are finitely generated
subobjects $X_1,X_2\subseteq\kr f$ such that $X/X_i\in\cc S_i$,
$i=0,1$. Then $X_1+X_2$ is a finitely generated subobject of $\kr f$
and $X/(X_1+X_2)\in\cc S_1\cap\cc S_2$. By~\cite[5.8]{G} $\cc
S_1\cap\cc S_2$ is of finite type.

By Lemma~\ref{ggg} $\surd\cup_{i\in I}\cc S_i$ is of finite type
with each $\cc S_i$ of finite type. It follows from
Proposition~\ref{tp} that $\cup_{i\in I}O(\cc S_i)=O(\cup_{i\in
I}\cc S_i)=O(\surd\cup_{i\in I}\cc S_i)$.

It follows from Proposition~\ref{tp} that the map~\eqref{111} is
bijective.
\end{proof}

Let $\latl(\cc C)$ denote the lattice of localizing subcategories of
$\cc C$, where, by definition,
   $$\cc S\wedge\cc Q=\cc S\cap\cc Q,\quad\cc S\vee\cc Q=\surd(\cc S\cup\cc Q)$$
for any $\cc S,\cc Q\in\latl(\cc C)$. The proof of Theorem~\ref{zg}
shows that the subset of localizing subcategories of finite type in
$\latl(\cc C)$ is a sublattice. We shall denote it by $\latfl(\cc
C)$.

\begin{rem}{\rm
If $\cc C$ is a locally coherent Grothendieck category, the
topological space $\inj_{fl}\cc C$ is also called in literature the
{\it Ziegler spectrum\/} of $\cc C$. It arises from the Ziegler work
on the model theory of modules~\cite{Z}. According to the original
Ziegler definition the points of the Ziegler spectrum of a ring $R$
are the isomorphism classes of indecomposable pure-injective right
$R$-modules. These can be identified with $\inj(\lfp,\text {Ab})$,
where $(\lfp,\text {Ab})$ is the locally coherent Grothendieck
category consisting of additive covariant functors defined on the
category of finitely presented left modules $\lfp$ with values in
the category of abelian groups Ab. The closed subsets correspond to
complete theories of modules. Later Herzog~\cite{H} and
Krause~\cite{Kr1} defined the Ziegler topology for arbitrary locally
coherent Grothendieck categories.

}\end{rem}

\begin{prop}\label{lenn}
Given a localizing subcategory of strictly finite type $\cc S$ in a
locally finitely presented Grothendieck category $\cc C$, the
topology on $\inj_{fl}(\cc C/\cc S)$ coincides with the subspace
topology induced by $\inj_{fl}\cc C$.
\end{prop}

\begin{proof}
By~\cite[5.9]{G} $\cc C/\cc S$ is a locally finitely presented
Grothendieck category, and so the fl-topology on $\inj(\cc C/\cc S)$
makes sense. Let $O(\cc P)$ be an open subset of $\inj_{fl}(\cc
C/\cc S)$ with $\cc P$ a localizing subcategory of finite type of
$\cc C/\cc S$. There is a unique localizing subcategory $\cc T$ of
$\cc C$ such that $(\cc C/\cc S)/\cc P\cong\cc C/\cc T$. We claim
that $\cc T$ is of finite type.

It is plainly enough to verify that $X=\sum_{\cc C}X_\alpha$ is a
$\cc T$-closed object whenever each $X_\alpha$ is $\cc T$-closed.
Since $\cc S$ and $\cc P$ are of finite type in $\cc C$ and $\cc
C/\cc S$ respectively, $X$ is both $\cc S$-closed and $\cc P$-closed
in $\cc C$ and $\cc C/\cc S$ respectively. It follows that $X$ is
$\cc T$-closed. Therefore $\cc T$ is of finite type and $O(\cc
P)=\inj_{fl}(\cc C/\cc S)\cap O(\cc T)$.

Now let $\cc Q$ be a localizing subcategory of finite type in $\cc
C$. We want to show that $\inj_{fl}(\cc C/\cc S)\cap O(\cc Q)$ is
open in $\inj_{fl}(\cc C/\cc S)$. Let $\wh{\cc Q}=\{X_{\cc S}\mid
X\in\cc Q\}$, then $\wh{\cc Q}$ is closed under direct sums,
subobjects, quotient objects in $\cc C/\cc S$ (see the proof of
Lemma~\ref{lenya}) and $O(\wh{\cc Q})=O(\surd\wh{\cc
Q})=\inj_{fl}(\cc C/\cc S)\cap O(\cc Q)$. We have to show that
$\surd\wh{\cc Q}$ is of finite type in $\cc C/\cc S$.

If we show that every direct limit $C=\lp_{\cc C/\cc S} C_\delta$ of
$\surd\wh{\cc Q}$-closed objects $C_\delta$ has no $\surd\wh{\cc
Q}$-torsion, it will follow from~\cite[5.8]{G} that $\surd\wh{\cc
Q}$ is of finite type. Obviously, each $C_\delta$ is $\cc Q$-closed.

Using Proposition~\ref{surd}, it is enough to check that $\Hom_{\cc
C}(X,C)=0$ for any object $X\in\wh{\cc Q}$. Since $\cc S$ is of
strictly finite type, one has $C\cong\lp_{\cc C} C_\delta$. Each
$C_\delta$ is $\cc Q$-closed, and therefore $\lp_{\cc C} C_\delta$
has no $\cc Q$-torsion by~\cite[5.8]{G} and the fact that $\cc Q$ is
of finite type. There is an object $Y\in\cc Q$ such that $Y_{\cc
S}=X$. Then $\Hom_{\cc C/\cc S}(X,C)\cong\Hom_{\cc C}(Y,\lp_{\cc C}
C_\delta)=0$, as required.
\end{proof}

\section{The topological space $\inj_{fl,\otimes}(X)$}

In the preceding section we studied some general properties of
finite localizations in locally finitely presented Grothendieck
categories and their relation with the topological space
$\inj_{fl}\cc C$. In this section we introduce and study the
topological space $\inj_{fl,\otimes}(X)$ which is of particular
importance in practice. If otherwise specified, $X$ is supposed to
be quasi-compact and quasi-separated.

Given a quasi-compact open subset $U\subset X$, we denote by $\cc
S_U=\{\cc F\in\Qcoh(X)\mid \cc F|_U=0\}$. It follows
from~\cite[5.9]{G} and the fact that $\cc F|_U\in\fp(\Qcoh(U))$
whenever $\cc F\in\fp(\Qcoh(X))$ that $\cc S_U$ is of strictly
finite type. Below we shall need the following

\begin{lem}\label{zeg}
Let $X$ be a quasi-compact and quasi-separated scheme and let $U,V$
be quasi-compact open subsets. Then the following relation holds:
   $$\cc S_{U\cap V}=\surd(\cc S_U\cup\cc S_V).$$
\end{lem}

\begin{proof}
Clearly, $\cc S_{U\cap V}$ contains both $\cc S_{U}$ and $\cc S_{V}$
and so $\cc S_{U\cap V}\supset\surd(\cc S_U\cup\cc S_V)$. Let $\cc
F\in\cc S_{U\cap V}$ and let $j:U\to X$ be the canonical inclusion.
Then $j_*j^*(\cc F)\in\cc S_V$. One has the following exact sequence
   $$0\to t_{\cc S_U}(\cc F)\to\cc F\lra{\lambda_{\cc F}}j_*j^*(\cc F).$$
Since $t_{\cc S_U}(\cc F)\in\cc S_U$ and $\im(\lambda_{\cc F})\in\cc
S_V$, we see that $\cc F\in\surd(\cc S_U\cup\cc S_V)$.
\end{proof}

We denote by $\inj_{fl}(X)$ the topological space
$\inj_{fl}(\Qcoh(X))$.

\begin{cor}\label{zhg}
Let $X$ be a quasi-compact and quasi-separated scheme and $X=U\cup
V$ with $U,V$ quasi-compact open subsets. Then the following
relations hold:
   \begin{gather*}
    \inj(X)=\inj(U)\cup\inj(V),\quad\inj(U\cap
    V)=\inj(U)\cap\inj(V)\\
    \inj_{fl}(X)=\inj_{fl}(U)\cup\inj_{fl}(V),\quad\inj_{fl}(U\cap V)=\inj_{fl}(U)\cap\inj_{fl}(V).
   \end{gather*}
\end{cor}

\begin{proof}
It follows from the fact that $\cc S_U\cap\cc S_V=0$,
Propositions~\ref{loc}, \ref{tp}, \ref{len}, \ref{lenn},
Theorem~\ref{zg}, and Lemma~\ref{zeg}.
\end{proof}

Let $\latl(X)$ (respectively, $\latfl(X)$) denote the lattice
$\latl(\Qcoh(X))$ (respectively, $\latfl(\Qcoh(X))$). It follows
from Proposition~\ref{tp} and Theorem~\ref{zg} that the map
$\latl(X)\to\lato(\inj(X))$ (respectively,
$\latfl(X)\to\lato(\inj_{fl}(X))$) is a lattice isomorphism. Suppose
$U\subset X$ is a quasi-compact open subset. Then the map
   $$\alpha_{X,U}:\latl(X)\to\latl(U),\quad\cc S\mapsto\surd(\wh{\cc S}|_U),$$
where $\wh{\cc S}|_U=\{\cc F|_U=\cc F_{\cc S_U}\mid\cc F\in\cc S\}$,
is a lattice map. If $V$ is another quasi-compact subset of $X$ such
that $X=U\cup V$ then, obviously,
   $$\alpha_{X,U\cap V}=\alpha_{U,U\cap V}\circ\alpha_{X,U}=\alpha_{V,U\cap V}\circ\alpha_{X,V}.$$
By the proof of Proposition~\ref{lenn} $\alpha_{X,U}(\cc
S)\in\latfl(U)$ for every $\cc S\in\latfl(U)$. Thus we have a map
   $$\alpha_{X,U}:\latfl(X)\to\latfl(U).$$

The notion of a pullback for lattices satisfying the obvious
universal property is easily defined.

\begin{lem}\label{pull}
The commutative squares of lattices
   $$\begin{CD}
      \latl(X)@>\alpha_{X,U}>>\latl(U)\\
      @V\alpha_{X,V}VV@VV{\alpha_{U,U\cap V}}V\\
      \latl(V)@>\alpha_{V,U\cap V}>>\latl(U\cap V)
     \end{CD}$$
and
   $$\begin{CD}
      \latfl(X)@>\alpha_{X,U}>>\latfl(U)\\
      @V\alpha_{X,V}VV@VV{\alpha_{U,U\cap V}}V\\
      \latfl(V)@>\alpha_{V,U\cap V}>>\latfl(U\cap V)
     \end{CD}$$
are pullback.
\end{lem}

\begin{proof}
It is enough to observe that the commutative squares of lemma are
isomorphic to the corresponding pullback squares of lattices of open
sets
   $$\begin{CD}
      \lato(\inj(X))@>\alpha_{X,U}>>\lato(\inj(U))\\
      @V\alpha_{X,V}VV@VV{\alpha_{U,U\cap V}}V\\
      \lato(\inj(V))@>\alpha_{V,U\cap V}>>\lato(\inj(U\cap V))
     \end{CD}$$
and
   $$\begin{CD}
      \lato(\inj_{fl}(X))@>\alpha_{X,U}>>\lato(\inj_{fl}(U))\\
      @V\alpha_{X,V}VV@VV{\alpha_{U,U\cap V}}V\\
      \lato(\inj_{fl}(V))@>\alpha_{V,U\cap V}>>\lato(\inj_{fl}(U\cap V))
     \end{CD}$$
(see Proposition~\ref{tp}, Theorem~\ref{zg}, and
Corollary~\ref{zhg}).
\end{proof}

Recall that $\Qcoh(X)$ is monoidal with the tensor product
$\otimes_X$ right exact and preserving direct limits (see~\cite[\S
II.2]{KS}).

\begin{defs}{\rm
A localizing subcategory $\cc S$ of $\Qcoh(X)$ is said to be {\it
tensor\/} if $\cc F\otimes_X\cc G\in\cc S$ for every $\cc F\in\cc S$
and $\cc G\in\Qcoh(X)$.

}\end{defs}

\begin{lem}\label{pp}
A localizing subcategory of finite type $\cc S\subset\Qcoh(X)$ is
tensor if and only if $\cc F\otimes_X\cc G\in\cc S$ for every $\cc
F\in\cc S\cap\fp(\Qcoh(X))$ and $\cc G\in\fp(\Qcoh(X))$.
\end{lem}

\begin{proof}
It is enough to observe that every $\cc F\in\cc S$ is a quotient
object of the direct sum of objects from $\cc S\cap\fp(\Qcoh(X))$
and that every object $\cc G\in\Qcoh(X)$ is a direct limit of
finitely presented objects.
\end{proof}

\begin{lem}\label{ppp}
Let $\cc X\subset\Qcoh(X)$ be a subcategory closed under direct
sums, subobjects, quotient objects, and tensor products. Then
$\surd\cc X$ is a tensor localizing subcategory.
\end{lem}

\begin{proof}
By Proposition~\ref{surd} an object $\cc F\in\surd\cc X$ if and only
if there is a filtration
   $$\cc F_0\subset\cc F_1\subset\cdots\subset\cc F_\beta\subset\cdots$$
such that $\cc F=\bigcup_\beta\cc F_\beta$, $\cc
F_\gamma=\bigcup_{\beta<\gamma}\cc F_\beta$ if $\gamma$ is a limit
ordinal, and $\cc F_0,\cc F_{\beta+1}/\cc F_\beta\in\langle\cc
X^\ps\rangle=\cc X$.

We have $\cc F_0\otimes_X\cc G\in\cc X$ for any $\cc G\in\Qcoh(X)$.
Suppose $\beta=\alpha+1$ and $\cc F_\alpha\otimes_X\cc G\in\surd\cc
X$. One has an exact sequence
   $$\cc F_\alpha\otimes_X\cc G\lra{f}\cc F_\beta\otimes_X\cc G\to(\cc F_\beta/\cc F_\alpha)\otimes_X\cc G\to 0.$$
Since $(\cc F_\beta/\cc F_\alpha)\otimes_X\cc G\in\cc X$ and $\im
f\in\surd\cc X$, we see that $\cc F_\beta\otimes_X\cc G\in\surd\cc
X$. If $\gamma$ is a limit ordinal and $\cc F_\beta\otimes_X\cc
G\in\surd\cc X$ for all $\beta<\gamma$, then $\cc
F_\gamma\otimes_X\cc G=\lp_{\beta<\gamma}(\cc F_\beta\otimes_X\cc
G)\in\surd\cc X$. Therefore $\surd\cc X$ is tensor.
\end{proof}

The next statement is of great utility in this paper.

\begin{red}
Let $\ff S$ be the class of quasi-compact, quasi-separated schemes
and let $P$ be a property satisfied by some schemes from $\ff S$.
Assume in addition the following.
\begin{enumerate}
\item $P$ is true for affine schemes.
\item If $X\in\ff S$, $X=U\cup V$, where $U,V$ are quasi-compact open subsets of $X$, and
$P$ holds for $U,V,U\cap V$ then it holds for $X$.
\end{enumerate}
Then $P$ holds for all schemes from $\ff S$.
\end{red}

\begin{proof}
See the proof of \cite[3.9.2.4]{Lipman} and \cite[3.3.1]{BV}.
\end{proof}

\begin{lem}\label{qqq}
The join $\cc T=\surd(\cc S\cup\cc Q)$ of two tensor localizing
subcategories $\cc S,\cc Q\subset\Qcoh(X)$ is tensor.
\end{lem}

\begin{proof}
We use the Reduction Principle to demonstrate the lemma. It is true
for affine schemes, because every localizing subcategory is tensor
in this case. Suppose $X=U\cup V$, where $U,V$ are quasi-compact
open subsets of $X$, and the assertion is true for $U,V,U\cap V$. We
have to show that it is true for $X$ itself.

We have the following relation:
   $$\alpha_{X,U}(\cc T)=(\surd(\wh{\cc S}|_U))\vee(\surd(\wh{\cc Q}|_U)).$$
Given $\cc F,\cc G\in\Qcoh(X)$ there is a canonical isomorphism
(see~\cite[II.2.3.5]{KS})
   $$(\cc F|_U)\otimes_U(\cc G|_U)\cong(\cc F\otimes_X\cc G)|_U.$$
It follows that both $\wh{\cc S}|_U$ and $\wh{\cc Q}|_U$ are closed
under tensor products. By Lemma~\ref{ppp} both $\surd\wh{\cc S}|_U$
and $\surd\wh{\cc Q}|_U$ are tensor. By assumption, the join of two
tensor localizing subcategories in $\Qcoh(U)$ is tensor, and so
$\alpha_{X,U}(\cc T)$ is tensor. For the same reasons,
$\alpha_{X,V}(\cc T)$ is tensor. Obviously, $\cc T$ is tensor
whenever so are $\alpha_{X,U}(\cc T)$ and $\alpha_{X,V}(\cc T)$.
Therefore $\cc T$ is tensor as well and our assertion now follows
from the Reduction Principle.
\end{proof}

Given a tensor localizing subcategory of finite type $\cc S$ in
$\Qcoh(X)$, we denote by
   $$\cc O(\cc S)=\{\cc E\in\inj(X)\mid t_{\cc S}(\cc E)\neq 0\}.$$

\begin{thm}\label{tenfg}
The collection of subsets of the injective spectrum $\inj(X)$,
   $$\{\cc O(\cc S)\mid\cc S\subset\cc C \textrm{ is a tensor localizing subcategory of finite type}\},$$
satisfies the axioms for the open sets of a topology on $\inj(X)$.
This topological space will be denoted by $\inj_{fl,\otimes}(X)$ and
this topology will be referred to as the {\em tensor fl-topology}.
Moreover, the map
   \begin{equation}\label{1111}
    \cc S\longmapsto\cc O(\cc S)
   \end{equation}
is an inclusion-preserving bijection between the tensor localizing
subcategories $\cc S$ of finite type in $\Qcoh(X)$ and the open
subsets of $\inj_{fl,\otimes}(X)$.
\end{thm}

\begin{proof}
Obviously, the intersection $\cc S_1\cap\cc S_2$ of two tensor
localizing subcategories of finite type is a tensor localizing
subcategory of finite type, hence $\cc O(\cc S_1\cap\cc S_2)=\cc
O(\cc S_1)\cap\cc O(\cc S_2)$ by Theorem~\ref{zg}.

Now let us show that the join $\cc T=\surd(\cup_{i\in I}\cc S_i)$ of
tensor localizing subcategories of finite type $\cc S_i$ is tensor.
By induction and Lemma~\ref{qqq} $\cc T$ is tensor whenever $I$ is
finite. Now we may assume, without loss of generality, that $I$ is a
directed index set and $\cc S_i\subset\cc S_j$ for any $i\le j$. By
the proof of Lemma~\ref{ggg} $\cc T=\surd\cc X$ with $\cc X$ the
full subcategory of $\Qcoh(X)$ of those objects which can be
presented as directed sums $\sum F_\alpha$ with each $F_\alpha$
belonging to $\cup_{i\in I}\cc S_i$. Then $\cc X$ is closed under
direct sums, subobjects and quotient objects. It is also closed
under tensor product, because $\otimes_X$ commutes with direct
limits. By Lemma~\ref{ppp} $\cc T$ is tensor and by Lemma~\ref{ggg}
$\cc T$ is of finite type.

Theorem~\ref{zg} implies that $\cc O(\surd(\cup_{i\in I}\cc
S_i))=\cup_{i\in I}\cc O(\cc S_i)$ and the map~\eqref{1111} is
bijective.
\end{proof}

We denote by $\lattfl(X)$ the lattice of tensor localizing
subcategories of finite type in $\Qcoh(X)$.

\begin{cor}\label{pullf}
The commutative square of lattices
   $$\begin{CD}
      \lattfl(X)@>\alpha_{X,U}>>\lattfl(U)\\
      @V\alpha_{X,V}VV@VV{\alpha_{U,U\cap V}}V\\
      \lattfl(V)@>\alpha_{V,U\cap V}>>\lattfl(U\cap V)
     \end{CD}$$
is pullback.
\end{cor}

\begin{proof}
The proof is similar to that of Lemma~\ref{pull}.
\end{proof}

\section{The classification theorem}

Recall from~\cite{Hoc} that a topological space is {\it spectral\/}
if it is $T_0$, quasi-compact, if the quasi-compact open subsets are
closed under finite intersections and form an open basis, and if
every non-empty irreducible closed subset has a generic point. Given
a spectral topological space, $X$, Hochster~\cite{Hoc} endows the
underlying set with a new, ``dual", topology, denoted $X^*$, by
taking as open sets those of the form $Y=\bigcup_{i\in\Omega}Y_i$
where $Y_i$ has quasi-compact open complement $X\setminus Y_i$ for
all $i\in\Omega$. Then $X^*$ is spectral and $(X^*)^*=X$ (see
\cite[Prop.~8]{Hoc}).

As an example, the underlying topological space of a quasi-compact,
quasi-separated scheme $X$ is spectral. In this section we shall
show that the tensor localizing subcategories of finite type in
$\Qcoh(X)$ are in 1-1 correspondence with the open subsets of $X^*$.
If otherwise specified, $X$ is supposed to be a quasi-compact,
quasi-separated scheme.

Given a quasi-compact open subset $D\subset X$, we denote by $\cc
S_D=\{\cc F\in\Qcoh(X)\mid\cc F|_D=0\}$.

\begin{prop}\label{qq}
Given an open subset $O=\cup_IO_i\subset X^*$, where each
$D_i=X\setminus O_i$ is quasi-compact and open in $X$, the
subcategory $\cc S=\{\cc F\in\Qcoh(X)\mid\supp_X(\cc F)\subseteq
O\}$ is a tensor localizing subcategory of finite type and $\cc
S=\surd(\cup_I\cc S_{D_i})$.
\end{prop}

\begin{proof}
Given a short exact sequence in $\Qcoh(X)$
   $$\cc F'\rightarrowtail\cc F\twoheadrightarrow\cc F'',$$
one has $\supp_X(\cc F)=\supp_X(\cc F')\cup\supp_X(\cc F'')$. It
follows that $\cc S$ is a Serre subcategory. It is also closed under
direct sums, hence localizing, because $\supp_X(\ps_I\cc
F_i)=\cup_I\supp_X(\cc F_i)$.

We use the Reduction Principle to show that $\cc S$ is a tensor
localizing subcategory of finite type and $\cc S=\surd(\cup_I\cc
S_{D_i})$. It is the case for affine schemes (see~\cite[2.2]{GP2}).
Suppose $X=U\cup V$, where $U,V$ are quasi-compact open subsets of
$X$, and the assertion is true for $U,V,U\cap V$. We have to show
that it is true for $X$ itself.

For any $\cc F\in\Qcoh(X)$ we have
   $$\supp_X(\cc F)=\supp_U(\cc F|_U)\cup\supp_V(\cc F|_V).$$
Clearly, $O\cap U$ is open in $U^*$ and $\supp_U(\cc F|_U)\subseteq
O\cap U$ for any $\cc F\in\cc S$. We see that $\wh{\cc S}|_U=\{\cc
F|_U=\cc F_{\cc S_U}\mid\cc F\in\cc S\}$ is contained in $\cc
S(U)=\{\cc F\in\Qcoh(U)\mid\supp_U(\cc F)\subseteq O\cap U\}$. By
assumption, $\cc S(U)$ is a tensor localizing subcategory of finite
type in $\Qcoh(U)$ and $\cc S(U)=\surd(\cup_I\cc S_{D_i\cap U})$. We
have $\cc S(U)\supset\surd(\wh{\cc S}|_U)$. Similarly, $\cc
S(V)\supset\surd(\wh{\cc S}|_V)$ and $\cc S(V)=\surd(\cup_I\cc
S_{D_i\cap V})$.

Since
   \begin{gather*}
     \alpha_{U,U\cap V}(\cc S_{D_i\cap U})=\alpha_{V,U\cap V}(\cc
     S_{D_i\cap V})\bl{\textrm{Lem.~\ref{zeg}}}=\cc S_{D_i\cap U\cap V}\\
     \bl{\textrm{Prop.~\ref{loc}}}=\{\cc F\in\Qcoh(U\cap V)\mid\supp_{U\cap V}(\cc
     F)\subseteq O_i\cap U\cap V\},
   \end{gather*}
it follows that
   \begin{gather*}
    \alpha_{U,U\cap V}(\cc S(U))=\alpha_{V,U\cap V}(\cc
    S(V))=\\
    \cc S(U\cap V)=\{\cc F\in\Qcoh(U\cap V)\mid\supp_{U\cap V}(\cc F)\subseteq O\cap U\cap V\}=
    \surd(\cup_I\cc S_{D_i\cap U\cap V}).
   \end{gather*}
By Corollary~\ref{pullf} there is a unique tensor localizing
subcategory of finite type $\cc T\in\lattfl(X)$ such that
$\surd(\wh{\cc T}|_U)=\cc S(U)$, $\surd(\wh{\cc T}|_V)=\cc S(V)$ and
$\cc T=\surd(\cup_I\cc S_{D_i})$. By construction, $\supp_X(\cc
F)\subseteq O$ for all $\cc F\in\cc T$, and hence $\cc T\subseteq\cc
S$, $\surd(\wh{\cc T}|_U)\subseteq\surd(\wh{\cc S}|_U)$,
$\surd(\wh{\cc T}|_V)\subseteq\surd(\wh{\cc S}|_V)$. Therefore $\cc
S(V)=\surd(\wh{\cc S}|_V)$ and $\cc S(V)=\surd(\wh{\cc S}|_V)$. By
Corollary~\ref{pullf} we have $\cc S=\cc T$.
\end{proof}

Let $X=U\cup V$ with $U,V$ open, quasi-compact subsets. Then
$X^*=U^*\cup V^*$ and both $U^*$ and $V^*$ are closed subsets of
$X^*$. Let $Y\in\lato(X^*)$ then $Y=\cup_IY_i$ with each
$D_i:=X\setminus Y_i$ open, quasi-compact subset in $X$. Since each
$D_i\cap U$ is an open and quasi-compact subset in $U$, it follows
that $Y\cap U=\cup_I(Y_i\cap U)\in\lato(U^*)$. Then the map
   $$\beta_{X,U}:\lato(X^*)\to\lato(U^*),\quad Y\mapsto Y\cap U$$
is a lattice map. The lattice map
$\beta_{X,V}:\lato(X^*)\to\lato(V^*)$ is similarly defined.

\begin{lem}\label{pullo}
The square
   $$\begin{CD}
      \lato(X^*)@>\beta_{X,U}>>\lato(U^*)\\
      @V\beta_{X,V}VV@VV{\beta_{U,U\cap V}}V\\
      \lato(V^*)@>\beta_{V,U\cap V}>>\lato((U\cap V)^*)
     \end{CD}$$
is commutative and pullback.
\end{lem}

\begin{proof}
It is easy to see that the lattice maps
   $$Y\in\lato(X^*)\mapsto(Y\cap U,Y\cap V)\in\lato(U^*)\prod_{\lato((U\cap V)^*)}\lato(V^*)$$
and
   $$(Y_1,Y_2)\in\lato(U^*)\prod_{\lato((U\cap V)^*)}\lato(V^*)\mapsto Y_1\cup Y_2\in\lato(X^*)$$
are mutual inverses.
\end{proof}

\begin{lem}\label{ooo}
Given a subcategory $\cc X$ in $\Qcoh(X)$, we have
   $$\bigcup_{\cc F\in\cc X}\supp_X(\cc F)=\bigcup_{\cc F\in\surd\cc X}\supp_X(\cc F).$$
\end{lem}

\begin{proof}
Since $\supp_X(\ps_I\cc F_i)=\cup_I\supp_X(\cc F_i)$ and
$\supp_X(\cc F)=\supp_X(\cc F')\cup\supp_X(\cc F'')$ for any short
exact sequence $\cc F'\rightarrowtail\cc F\twoheadrightarrow\cc F''$
in $\Qcoh(X)$, we may assume that $\cc X$ is closed under
subobjects, quotient objects, and direct sums, i.e. $\cc
X=\langle\cc X^\ps\rangle$. If $\cc F=\sum_I F_i$ we also have
$\supp_X(\cc F)\subseteq\cup_I\supp_X(\cc F_i)$. Now our assertion
follows from Proposition~\ref{surd}.
\end{proof}

\begin{lem}\label{sss}
Given a tensor localizing subcategory of finite type $\cc
S\in\lattfl(X)$, the set
   $$Y=\bigcup_{\cc F\in\cc S}\supp_X(\cc F)$$
is open in $X^*$.
\end{lem}

\begin{proof}
We use the Reduction Principle to show that $Y\in\lato(X^*)$. It is
the case for affine schemes (see~\cite[2.2]{GP2}). Suppose $X=U\cup
V$, where $U,V$ are quasi-compact open subsets of $X$, and the
assertion is true for $U,V,U\cap V$. We have to show that it is true
for $X$ itself.

By Corollary~\ref{pullf} $\alpha_{X,U}(\cc S)=\surd(\wh{\cc
S}|_U)\in\lattfl(U)$ and $\alpha_{X,V}(\cc S)=\surd(\wh{\cc
S}|_V)\in\lattfl(V)$. By assumption,
   $$Y_1=\bigcup_{\cc F\in\surd\wh{\cc S}|_U}\supp_U(\cc F)\in\lato(U^*)$$
and
   $$Y_2=\bigcup_{\cc F\in\surd\wh{\cc S}|_V}\supp_V(\cc F)\in\lato(V^*).$$
By Lemma~\ref{ooo}
   $$Y_1=\bigcup_{\cc F\in\wh{\cc S}|_U}\supp_U(\cc F)=\bigcup_{\cc F\in\cc S}\supp_U(\cc F|_U)$$
and
   $$Y_2=\bigcup_{\cc F\in\wh{\cc S}|_V}\supp_V(\cc F)=\bigcup_{\cc F\in\cc S}\supp_V(\cc F|_V).$$
For every $\cc F\in\Qcoh(X)$ we have $\supp_X(\cc F)=\supp_U(\cc
F|_U)\cup\supp_V(\cc F|_V)$. Therefore $Y_1=Y\cap U$ and $Y_2=Y\cap
V$. By Lemma~\ref{pullo} $Y=Y_1\cup Y_2\in\lato(X^*)$.
\end{proof}

We are now in a position to prove the main result of the paper.

\begin{thm}[Classification; see Garkusha-Prest~\cite{GP2} for affine schemes]\label{class}
Let $X$ be a quasi-compact, quasi-separated scheme. Then the maps
   $$Y\bl{\phi_X}\longmapsto\mathcal S(Y)=\{\cc F\in\Qcoh(X)\mid\supp_X(\cc F)\subseteq Y\}$$
and
   $$\mathcal S\bl{\psi_X}\longmapsto Y(\cc S)=\bigcup_{\cc F\in\mathcal S}\supp_X(\cc F)$$ induce
bijections between
\begin{enumerate}
 \item the set of all subsets of
the form $Y=\bigcup_{i\in\Omega}Y_i$ with quasi-compact open
complement $X\setminus Y_i$ for all $i\in\Omega$; that is, the set
of all open subsets of $X^*,$

\item the set of all tensor localizing subcategories of finite type in $\Qcoh(X)$.
\end{enumerate}
Moreover, $\cc S(Y)=\surd(\cup_{i\in I}\cc S(Y_i))=\surd(\cup_{i\in
I}\cc S_{D_i})$, where $D_i=X\setminus Y_i$, $\cc S_{D_i}=\{\cc
F\in\Qcoh(X)\mid\cc F|_{D_i}=0\}$.
\end{thm}

\begin{proof}
By Proposition~\ref{qq} and Lemma~\ref{sss} $\phi_X(Y)\in\lattfl(X)$
and $\psi_X(\cc S)\in\lato(X^*)$. We have lattice maps
   $$\phi_X:\lato(X^*)\to\lattfl(X),\quad\psi_X:\lattfl(X)\to\lato(X^*).$$
We use the Reduction Principle to show that $\phi_X\psi_X=1$ and
$\psi_X\phi_X=1$. It is the case for affine schemes
(see~\cite[2.2]{GP2}). Suppose $X=U\cup V$, where $U,V$ are
quasi-compact open subsets of $X$, and the assertion is true for
$U,V,U\cap V$. We have to show that it is true for $X$ itself.

One has the following commutative diagram of lattices:
    $$\xymatrix@!0{
     &&\lato(V^*)\ar[rrrrrr]^{\beta_{V,U\cap V}}\ar'[d][dd] &&&&&& \lato((U\cap V)^*)\ar[dd]^{\phi_{U\cap V}}\\
     \lato(X^*)\ar[urr]\ar[rrrrrr]\ar[dd]_{\phi_X} &&&&&& \lato(U^*)\ar[urr]\ar[dd]\\
     &&\lattfl(V)\ar@{.>}'[r][rrrrrr]\ar'[d][dd]&&&&&& \lattfl(U\cap V)\ar[dd]^{\psi_{U\cap V}}\\
     \lattfl(X)\ar[rrrrrr]^{\alpha_{X,U}}\ar[urr]\ar[dd]_{\psi_X}&&&&&&\lattfl(U)\ar[urr]\ar[dd]\\
     &&\lato(V^*)\ar@{.>}'[r][rrrrrr]&&&&&& \lato((U\cap V)^*)\\
     \lato(X^*)\ar[rrrrrr]_{\beta_{X,U}}\ar[urr]&&&&&&\lato(U^*)\ar[urr] }$$
By assumption, all vertical arrows except $\phi_X,\psi_X$ are
bijections. Precisely, the maps $\phi_U,\psi_U$ (respectively
$\phi_V,\psi_V$ and $\phi_{U\cap V},\psi_{U\cap V}$) are mutual
inverses. Since each horizontal square is pullback (see
Corollary~\ref{pullf} and Lemma~\ref{pullo}), it follows that
$\phi_X,\psi_X$ are mutual inverses.

The fact that $\cc S(Y)=\surd(\cup_{i\in I}\cc
S(Y_i))=\surd(\cup_{i\in I}\cc S|_{D_i})$ is a consequence of
Propositions~\ref{loc} and~\ref{qq}. The theorem is proved.
\end{proof}

Denote by $\perf(X)$ the derived category of perfect complexes, the
homotopy category of those complexes of sheaves of $\cc O_X$-modules
which are locally quasi-isomorphic to a bounded complex of free $\cc
O_X$-modules of finite type. We say a thick triangulated subcategory
$\cc A\subset\perf(X)$ is a {\it tensor subcategory\/} if for each
object $E$ in $\perf(X)$ and each $A$ in $\cc A$, the derived tensor
product $E\otimes^L_X A$ is also in $\cc A$.

Let $E$ be a complex of sheaves of $\cc O_X$-modules. The {\it
cohomological support\/} of $E$ is the subspace
$\supph_X(E)\subseteq X$ of those points $x\in X$ at which the stalk
complex of $\cc O_{X,x}$-modules $E_x$ is not acyclic. Thus
$\supph_X(E)=\bigcup_{n\in\bb Z}\supp_X(H_n(E))$ is the union of the
supports in the classic sense of the cohomology sheaves of $E$.

We shall write $L_{\thick}(\perf(X))$ to denote the lattice of all
thick subcategories of $\perf(X)$.

\begin{thm}[Thomason~\cite{T}]\label{nn}
Let $X$ be a quasi-compact and quasi-separated scheme. The
assignments
   $$\cc T\in L_{\thick}(\perf(X))\bl\mu\longmapsto Y(\cc T)=\bigcup_{E\in\cc T}\supph_X(E)$$
and
   $$Y\in\lato(X^*)\bl\nu\longmapsto
     \cc T(Y)=\{E\in\perf(X)\mid\supph_X(E)\subseteq Y\}$$
are mutually inverse lattice isomorphisms.
\end{thm}

The next result says that there is a 1-1 correspondence between the
tensor thick subcategories of perfect complexes and the tensor
localizing subcategories of finite type of quasi-coherent sheaves.

\begin{thm}[see Garkusha-Prest~\cite{GP2} for affine schemes]\label{lll}
Let $X$ be a quasi-compact and quasi-separated scheme. The
assignments
   $$\cc T\in L_{\thick}(\perf(X))\bl\rho\longmapsto\cc S=\{\cc F\in\Qcoh(X)\mid\supp_X(\cc F)\subseteq Y(\cc T)\}$$
and
   $$\cc S\in\lattfl(X)\bl\tau\longmapsto
     \{E\in\perf(X)\mid H_n(E)\in\cc S\textrm{ for all $n\in\bb Z$}\}$$
are mutually inverse lattice isomorphisms.
\end{thm}

\begin{proof}
Consider the following diagram
    $$\xymatrix{
      &\lato(X^*)\ar[dl]_{\phi}\ar[dr]^{\nu}\\
      \lattfl(X)\ar@<.6ex>[rr]^{\tau}&&L_{\thick}(\perf(X)),\ar@<.6ex>[ll]^{\rho}
      }$$
in which $\phi,\nu$ are the lattice maps described in
Theorems~\ref{class} and~\ref{nn}. Using the fact that
$\supph_X(E)=\bigcup_{n\in\bb Z}\supp_X(H_n(E))$, $E\in\perf(X)$,
and Theorems~\ref{class}, \ref{nn} one sees that $\tau\phi=\nu$ and
$\rho\nu=\phi$. Then $\rho\tau=\rho\nu\phi^{-1}=\phi\phi^{-1}=1$ and
$\tau\rho=\tau\phi\nu^{-1}=\nu\nu^{-1}=1$.
\end{proof}

\section{The Zariski topology on $\inj(X)$}

We are going to construct two maps
   $$\alpha:X\to\inj(X)\quad\textrm{and}\quad\beta:\inj(X)\to X.$$
Given $P\in X$ there is an affine neighborhood $U=\spec R$ of $P$.
Let $E_P$ denote the injective hull of the quotient module $R/P$.
Then $E_P$ is an indecomposable injective $R$-module. By
Proposition~\ref{loc} $\Rfp$ can be regarded as the quotient
category $\Qcoh(X)/\cc S_U$, where $\cc S_U=\kr j^*_U$ with
$j_U:U\to X$ the canonical injection. Therefore
$j_{U,*}:\Rfp\to\Qcoh(X)$ takes injectives to injectives. We set
$\alpha(P)=j_{U,*}(E_P)\in\inj(X)$.

The definition of $\alpha$ does not depend on choice of the affine
neighborhood $U$. Indeed, let $P\in V=\spec S$ with $S$ a
commutative ring. Then $j_{U,*}(E_P)\cong j_{V,*}(E_P)\cong j_{U\cap
V,*}(E_P)$, hence these represent the same element in $\inj(X)$. We
denote it by $\cc E_P$.

Now let us define the map $\beta$. Let $X=\cup_{i=1}^nU_i$ with each
$U_i=\spec R_i$ an affine scheme and let $\cc E\in\inj(X)$. Then
$\cc E$ has no $\cc S_{U_i}$-torsion for some $i\le n$, because
$\cap_{i=1}^n\cc S_{U_i}=0$ and $\cc E$ is uniform. Since $\Rfp_i$
is equivalent to $\Qcoh(X)/\cc S_{U_i}$, $\cc E$ can be regarded as
an indecomposable injective $R_i$-module. Set $P=P(\cc E)$ to be the
sum of annihilator ideals in $R_i$ of non-zero elements,
equivalently non-zero submodules, of $\cc E$. Since $\cc E$ is
uniform the set of annihilator ideals of non-zero elements of $\cc
E$ is closed under finite sum. It is easy to check (\cite[9.2]{Pr2})
that $P(\cc E)$ is a prime ideal. By construction, $P(\cc E)\in
U_i$. Clearly, the definition of $P(\cc E)$ does not depend on
choice of $U_i$ and $P(\cc E_P)=P$. We see that $\beta\alpha=1_X$.
In particular, $\alpha$ is an embedding of $X$ into $\inj(X)$. We
shall consider this embedding as identification.

Given a commutative coherent ring $R$ and an indecomposable
injective $R$-module $E\in\inj R$, Prest~\cite[9.6]{Pr2} observed
that $E$ is elementary equivalent to $E_{P(E)}$ in the first order
language of modules. Translating this fact from model-theoretic
idioms to algebraic language, it says that every localizing
subcategory of finite type $\cc S\in\latfl(\Rfp)$ is cogenerated by
prime ideals. More precisely, there is a set $D\subset\spec R$ such
that $S\in\cc S$ if and only if $\Hom_R(S,E_P)=0$ for all $P\in D$.
This has been generalized to all commutative rings by
Garkusha-Prest~\cite{GP2}. Moreover, $D=\spec
R\setminus\bigcup_{S\in\cc S}\supp_R(S)$.

\begin{prop}\label{iii}
Let $\cc E\in\inj(X)$ and let $P(\cc E)\in X$ be the point defined
above. Then $\cc E$ and $\cc E_{P(\cc E)}$ are topologically
indistinguishable in $\inj_{fl}(X)$. In other words, for every $\cc
S\in\latfl(X)$ the sheaf $\cc E$ has no $\cc S$-torsion if and only
if $\cc E_{P(\cc E)}$ has no $\cc S$-torsion.
\end{prop}

\begin{proof}
Let $U=\spec R\subset X$ be such that $\cc E$ has no $\cc
S_U$-torsion. Then $P(\cc E)\in U$ and $\cc E$ and $\cc E_{P(\cc
E)}$ have no $\cc S_U$-torsion. These can also be considered as
indecomposable injective $R$-modules, because $\Qcoh(X)/\cc
S_U\cong\Rfp$ by Proposition~\ref{loc}. Denote by $\cc
S'=\alpha_{X,U}(\cc S)\in\latfl(\Rfp)$. Then $\cc E$ has no $\cc
S$-torsion in $\Qcoh(X)$ if and only if $\cc E$ has no $\cc
S'$-torsion in $\Rfp$. Our assertion now follows
from~\cite[3.5]{GP2}.
\end{proof}

\begin{cor}\label{rrr}
If $\cc S\in\lattfl(X)$ then $\cc O(\cc S)\bigcap X=Y(\cc S)$, where
$Y(\cc S)=\bigcup_{\cc F\in\cc S}\supp_X(\cc F)\in\lato(X^*)$.
\end{cor}

\begin{proof}
If $\cc E\in\cc O(\cc S)$ and $U=\spec R\subset X$ is such that $\cc
E$ has no $\cc S_U$-torsion, then $P(\cc E)\in U$ and $\cc E_{P(\cc
E)}\in\cc O(\cc S)$ by Proposition~\ref{iii}. Let $\cc
S'=\alpha_{X,U}(\cc S)\in\latfl(\Rfp)$. We have $Y(\cc
S'):=\bigcup_{S\in\cc S}\supp_R(S)\subset Y(\cc S)$. By the proof of
Proposition~\ref{iii} $\cc E_{P(\cc E)}$ has $\cc S'$-torsion. Then
there is a finitely generated ideal $I\subset R$ such that
$R/I\in\cc S'$ and $\Hom_R(R/I,\cc E_{P(\cc E)})\neq 0$. It follows
from~\cite[3.4]{GP2} that $P(\cc E)\in Y(\cc S')$, and hence $\cc
O(\cc S)\bigcap X\subset Y(\cc S)$.

Conversely, if $P\in Y(\cc S)\cap U$ then $\cc E_P$ has $\cc
S'$-torsion by~\cite[3.4]{GP2}. Therefore $\cc E_P\in\cc O(\cc
S')\subset\cc O(\cc S)$. It immediately follows that $\cc O(\cc
S)\bigcap X\supset Y(\cc S)$.
\end{proof}

\begin{prop} \label{bbb} (cf.~\cite[3.7]{GP2}) 
Let $X$ be a quasi-compact and quasi-separated scheme. Then the maps
   $$Y\in\lato(X^*)\bl\sigma\mapsto\cc O_{Y}=\{\cc E\in\inj(X)\mid P(\cc E)\in Y\}$$
and
   $$\cc O\in\lato(\inj_{fl,\otimes}(X))\bl\varepsilon\mapsto Y_{\cc O}=\{P(\cc E)\in X^*\mid\cc E\in\cc O\}=
     \cc O\cap X^*$$
induce a 1-1 correspondence between the lattices of open sets of
$X^*$ and those of $\inj_{fl,\otimes}(X)$.
\end{prop}

\begin{proof}
Let $\cc S(Y)=\{\cc F\in\Qcoh(X)\mid\supp_X(\cc F)\subseteq
Y\}\in\lattfl(X)$; then $Y=Y(\cc S(Y))$ by the Classification
Theorem and $\cc O(\cc S(Y))\bigcap X=Y$ by Corollary~\ref{rrr}. It
follows that $\cc O(\cc S(Y))\subseteq\cc O_Y$. On the other hand,
if $\cc E\in\cc O_{Y}$ then the proof of Corollary~\ref{rrr} shows
that $\cc E_{P(\cc E)}\in\cc O(\cc S(Y))$. Proposition~\ref{iii}
implies $\cc E\in\cc O(\cc S(Y))$, hence $\cc O(\cc
S(Y))\supseteq\cc O_Y$. We see that $\cc O_Y=\cc O(\cc
S(Y))\in\lato(\inj_{fl,\otimes}(X))$.

Let $\cc O\in\lato(\inj_{fl,\otimes}(X))$. By Theorem~\ref{tenfg}
there is a unique $\cc S\in\lattfl(X)$ such that $\cc O=\cc O(\cc
S)$. By Corollary~\ref{rrr} $\cc O\bigcap X=Y(\cc S)=Y_{\cc O}$, and
so $Y_{\cc O}\in\lato(X^*)$. It is now easy to verify that $Y_{\cc
O_{Y}}=Y$ and $\cc O_{Y_{\cc O}}=\cc O$.
\end{proof}

We notice that a subset $Y\subset X^*$ is open and quasi-compact in
$X^*$ if and only if $X\setminus Y$ is an open and quasi-compact
subset in $X$.

\begin{prop}\label{rr}
An open subset $\cc O\in\lato(\inj_{fl,\otimes}(X))$ is
quasi-compact if and only if it is of the form $\cc O=\cc O(\cc
S(Y))$ with $Y$ an open and quasi-compact subset in $X^*$. The space
$\inj_{fl,\otimes}(X)$ is quasi-compact, the quasi-compact open
subsets are closed under finite intersections and form an open
basis, and every non-empty irreducible closed subset has a generic
point.
\end{prop}

\begin{proof}
Let $\cc O\in\lato(\inj_{fl,\otimes}(X))$ be quasi-compact. By
Theorem~\ref{tenfg} there is a unique $\cc S\in\lattfl(X)$ such that
$\cc O=\cc O(\cc S)$. By the Classification Theorem $\cc S=\cc
S(Y)=\surd(\cup_I\cc S(Y_i))$, where $Y=\bigcup_{\cc F\in\cc
S}\supp_X(\cc F)\in\lato(X^*)$, each $Y_i$ is such that $X\setminus
Y_i$ is open and quasi-compact subset of $X$, and $Y=\bigcup_IY_i$.
Then $\cc O=\bigcup_I\cc O(\cc S(Y_i))$. Since $\cc O$ is
quasi-compact, there is a finite subset $J\subset I$ such that $\cc
O=\bigcup_J\cc O(\cc S(Y_i))=\cc O(\surd(\cup_J\cc S(Y_i)))=\cc
O(\cc S(\cup_JY_i))$. Since $X$ is spectral, then
$X\setminus(\cup_JY_i)=\cap_J(X\setminus Y_i)$ is an open and
quasi-compact subset in $X$.

Conversely, let $\cc O=\cc O(\cc S(Y))$ with $X\setminus Y$ an open
and quasi-compact subset in $X$ and let $\cc O=\bigcup_I\cc O_i$
with each $\cc O_i\in\lato(\inj_{fl,\otimes}(X))$. By
Theorem~\ref{tenfg} there are unique $\cc S_i\in\lattfl(X)$ such
that $\cc O_i=\cc O(\cc S_i)$ and $\cc S(Y)=\surd(\cup_I\cc S_i)$.
We set $Y_i=\bigcup_{\cc F\in\cc S_i}\supp_X(\cc F)$ for each $i\in
I$. By Lemma~\ref{ooo} and the Classification Theorem one has
$Y=\bigcup_IY_i$. Since $Y$ is quasi-compact in $X^*$, there is a
finite subset $J\subset I$ such that $Y=\bigcup_JY_i$. It follows
that $\cc S(Y)=\surd(\cup_J\cc S_i)$ and $\cc O=\bigcup_J\cc O_i$.

The space $\inj_{fl,\otimes}(X)$ is quasi-compact, because it equals
$\cc O(\cc S(X^*))$ and $X^*$ is quasi-compact. The quasi-compact
open subsets are closed under finite intersections, because $\cc
O(\cc S(Y_1))\cap\cc O(\cc S(Y_2))=\cc O(\cc S(Y_1\cap Y_2))$ with
$Y_1,Y_2$ open and quasi-compact subsets in $X^*$. Since $\cc O(\cc
S(Y))=\cup_I\cc O(\cc S(Y_i))$, where $Y=\cup_IY_i$ and each $Y_i$
is an open and quasi-compact subset in $X^*$, the quasi-compact open
subsets also form an open basis.

Finally, it follows from Corollary~\ref{rrr} that a subset $U$ of
$\inj_{fl,\otimes}(X)$ is closed and irreducible \ifff so is $\wh
U:=U\cap X^*$. Since $X^*$ is spectral then $\wh U$ has a generic
point $P$. The point $\cc E_P\in U$ is generic.
\end{proof}

Though the space $\inj_{fl,\otimes}(X)$ is not in general $T_0$
(see~\cite{GP}), nevertheless we make the same definition for
$(\inj_{fl,\otimes}(X))^*$ as for spectral spaces and denote it by
$\inj_{zar}(X)$. By definition, $\cc Q\in\lato(\inj_{zar}(X))$ \ifff
$\cc Q=\cup_I\cc Q_i$ with each $\cc Q_i$ having quasi-compact and
open complement in $\inj_{fl,\otimes}(X)$. The topology on
$\inj_{zar}(X)$ will also be referred to as the {\it Zariski
topology}. Notice that the Zariski topology on $\inj_{zar}(\spec
R)$, $R$ is coherent, concides with the Zariski topology on the
injective spectrum $\inj R$ in the sense of Prest~\cite{Pr2}.

\begin{thm}[cf.~Garkusha-Prest~\cite{GP,GP1,GP2})]\label{bmb}
Let $X$ be a quasi-compact and quasi-separated scheme. The space $X$
is dense and a retract in $\inj_{zar}(X)$. A left inverse to the
embedding $X\hookrightarrow\inj_{zar}(X)$ takes $\cc
E\in\inj_{zar}(X)$ to $P(\cc E)\in X$. Moreover, $\inj_{zar}(X)$ is
quasi-compact, the basic open subsets $\cc Q$, with
$\inj(X)\setminus\cc Q$ quasi-compact and open subset in
$\inj_{fl,\otimes}(X)$, are quasi-compact, the intersection of two
quasi-compact open subsets is quasi-compact, and every non-empty
irreducible closed subset has a generic point.
\end{thm}

\begin{proof}
Let $\cc Q\in\lato(\inj_{zar}(X))$ be such that $\cc
O:=\inj(X)\setminus\cc Q$ is a quasi-compact and open subset in
$\inj_{fl,\otimes}(X)$ and let $Y=\cc O\cap X$ and $D=X\setminus
Y=\cc Q\cap X$. Since $Y$ is a quasi-compact subset in $X^*$, then
$D$ is a quasi-compact subset in $X$. Notice that $\cc O=\cc O(\cc
S_D)$, where $\cc S_D=\{\cc F\in\Qcoh(X)\mid\cc F|_D=0\}$. Clearly,
$X$ is dense in $\inj_{zar}(X)$ and $\alpha:X\to\inj_{zar}(X)$ is a
continuous map.

The map $\beta:\inj_{zar}(X)\to X$, $\cc E\mapsto P(\cc E)$, is left
inverse to $\alpha$. Obviously, $\beta$ is continuous. Thus $X$ is a
retract of $\inj_{zar}(X)$.

Let us show that the basic open set $\cc Q$ is quasi-compact. Let
$\cc Q=\bigcup_{i\in\Omega}\cc Q_i$ with each $\inj(X)\setminus\cc
Q_i$ a quasi-compact and open subset in $\inj_{fl,\otimes}(X)$ and
$D_i:=\cc Q_i\cap X$. Since $D$ is quasi-compact, then
$D=\bigcup_{i\in\Omega_0}D_i$ for some finite subset
$\Omega_0\subset\Omega$.

Assume $\cc E\in\cc Q\setminus\bigcup_{i\in\Omega_0}\cc Q_i$. It
follows from Proposition~\ref{bbb} that ${P(\cc E)}\in\cc Q\cap
X=D=\bigcup_{i\in\Omega_0}D_i$. Proposition~\ref{bbb} implies that
$\cc E\in\cc Q_{i_0}$ for some $i_0\in\Omega_0$, a contradiction. So
$\cc Q$ is quasi-compact. It also follows that the intersection of
two quasi-compact open subsets is quasi-compact and that
$\inj_{zar}(X)$ is quasi-compact.

Finally, it follows from Corollary~\ref{rrr} that a subset $U$ of
$\inj_{zar}(X)$ is closed and irreducible \ifff so is $\wh U:=U\cap
X$. Since $X$ is spectral then $\wh U$ has a generic point $P$. The
point $\cc E_P\in U$ is generic.
\end{proof}

\begin{cor}\label{eme}
Let $X$ be a quasi-compact and quasi-separated scheme. The following
relations hold:
   $$\inj_{zar}(X)=(\inj_{fl,\otimes}(X))^* \textrm{ and }\inj_{fl,\otimes}(X)=(\inj_{zar}(X))^*.$$
\end{cor}

Though the space $\inj_{zar}(X)$ is strictly bigger than $X$ in
general (see~\cite{GP}), their lattices of open subsets are
isomorphic. More precisely, Proposition~\ref{bbb} implies that the
maps
   $$D\in\lato(X)\mapsto\cc Q_{D}=\{\cc E\in\inj(X)\mid P(\cc E)\in D\}$$
and
   $$\cc Q\in\lato(\inj_{zar}(X))\mapsto D_{\cc Q}=\{P(\cc E)\in X\mid\cc E\in\cc Q\}=
     \cc Q\cap X$$
induce a 1-1 correspondence between the lattices of open sets of $X$
and those of $\inj_{zar}(X)$. Moreover, sheaves do not see any
difference between $X$ and $\inj_{zar}(X)$. Namely, the following is
true.

\begin{prop}\label{spi}
Let $X$ be a quasi-compact and quasi-separated scheme. Then the maps
of topological spaces $\alpha:X\to\inj_{zar}(X)$ and
$\beta:\inj_{zar}(X)\to X$ induce isomorphisms of the categories of
sheaves
   $$\beta_*:Sh(\inj_{zar}(X))\lra{\cong}Sh(X),\quad\alpha_*:Sh(X)\lra{\cong}Sh(\inj_{zar}(X)).$$
\end{prop}

\begin{proof}
Since $\beta\alpha=1$ it follows that $\beta_*\alpha_*=1$. By
definition, $\beta_*(\cc F)(D)=\cc F(\cc Q_D)$ for any $\cc F\in
Sh(\inj_{zar}(X)), D\in\lato(X)$ and $\alpha_*(\cc G)(\cc Q)=\cc
G(D_{\cc Q})$ for any $\cc G\in Sh(X), \cc
Q\in\lato(\inj_{zar}(X))$. We have:
   $$\alpha_*\beta_*(\cc F)(\cc Q)=\beta_*(\cc F)(D_{\cc Q})=\cc F(\cc Q_{D_{\cc Q}})=\cc F(\cc Q).$$
We see that $\alpha_*\beta_*=1$, and so $\alpha_*,\beta_*$ are
mutually inverse isomorphisms.
\end{proof}

Let $\cc O_{\inj_{zar}(X)}$ denote the sheaf of commutative rings
$\alpha_*(\cc O_X)$; then $(\inj_{zar}(X),\cc O_{\inj_{zar}(X)})$ is
plainly a locally ringed space. If we set $\alpha^\sharp:\cc
O_{\inj_{zar}(X)}\to\alpha_*\cc O_X$ and $\beta^\sharp:\cc
O_X\to\beta_*\cc O_{\inj_{zar}(X)}$ to be the identity maps, then
the map of locally ringed spaces
   $$(\alpha,\alpha^\sharp):(X,\cc O_X)\to(\inj_{zar}(X),\cc O_{\inj_{zar}(X)})$$
is right inverse to
   $$(\beta,\beta^\sharp):(\inj_{zar}(X),\cc O_{\inj_{zar}(X)})\to(X,\cc O_X).$$
Observe that it is {\it not\/} a scheme in general, because
$\inj_{zar}(X)$ is not a $T_0$-space. Proposition~\ref{spi} implies
that the categories of the $\cc O_{\inj_{zar}(X)}$-modules and $\cc
O_{X}$-modules are naturally isomorphic.

\section{The prime spectrum of an ideal lattice}

Inspired by recent work of Balmer~\cite{B}, Buan, Krause, and
Solberg~\cite{BKS} introduce the notion of an ideal lattice and
study its prime ideal spectrum. Applications arise from abelian or
triangulated tensor categories.

\begin{defs}[Buan, Krause, Solberg~\cite{BKS}] {\rm
An {\it ideal lattice\/} is by definition a partially ordered set
$L=(L,\leq)$, together with an associative multiplication $L\times
L\to L$, such that the following holds.
\begin{enumerate}
\item[(L1)] The poset $L$ is a {\it complete lattice\/}, that is,
$$\sup A = \bigvee_{a\in A} a\quad\text{and}\quad \inf A = \bigwedge_{a\in A}
a$$ exist in $L$ for every subset $A\subseteq L$.
\item[(L2)] The lattice $L$ is {\it compactly generated\/}, that is,
every element in $L$ is the supremum of a set of compact elements.
(An element $a\in L$ is {\em compact}, if for all $A\subseteq L$
with $a\leq \sup A$ there exists some finite $A'\subseteq A$ with
$a\leq\sup A'$.)
\item[(L3)] We have for all $a,b,c\in L$
$$a(b\vee c)=ab\vee ac\quad\text{and}\quad (a\vee b)c=ac\vee bc.$$
\item[(L4)] The element $1=\sup L$ is compact, and $1a=a=a1$ for all $a\in L$.
\item[(L5)] The product of two compact elements is again compact.
\end{enumerate}
A {\it morphism\/} $\phi\colon L\to L'$ of ideal lattices is a map
satisfying
\begin{gather*}\label{eq:mor}
\phi(\bigvee_{a\in A}a)=\bigvee_{a\in A}\phi(a)\quad \text{for}\quad
A\subseteq L, \\
\phi(1)=1\quad\text{and}\quad\phi(ab)=\phi(a)\phi(b)\quad\text{for}\quad
a,b\in L.\notag
\end{gather*}
}\end{defs}

Let $L$ be an ideal lattice. Following~\cite{BKS} we define the
spectrum of prime elements in $L$. An element $p\neq 1$ in $L$ is
 {\em prime} if $ab\leq p$ implies $a\leq p$ or $b\leq p$ for
all $a,b\in L$. We denote by $\spec L$ the set of prime elements in
$L$ and define for each $a\in L$
   $$V(a)=\{p\in\spec L\mid a\leq p\}\quad\text{and}\quad D(a)=\{p\in\spec L\mid a\not\leq p\}.$$
The subsets of $\spec L$ of the form $V(a)$ are closed under forming
arbitrary intersections and finite unions.  More precisely,
   $$V(\bigvee_{i\in\Omega} a_i)=\bigcap_{i\in\Omega} V(a_i)\quad\text{and}\quad
   V(ab)=V(a)\cup V(b).$$
Thus we obtain the {\it Zariski topology\/} on $\spec L$ by
declaring a subset of $\spec L$ to be {\it closed\/} if it is of the
form $V(a)$ for some $a\in L$. The set $\spec L$ endowed with this
topology is called the {\it prime spectrum\/} of $L$.  Note that the
sets of the form $D(a)$ with compact $a\in L$ form a basis of open
sets. The prime spectrum $\spec L$ of an ideal lattice $L$ is
spectral~\cite[2.5]{BKS}.

There is a close relation between spectral spaces and ideal
lattices. Given a topological space $X$, we denote by $L_{\open}(X)$
the lattice of open subsets of $X$ and consider the multiplication
map
   $$L_{\open}(X)\times L_{\open}(
 X)\to L_{\open}(X),\quad (U,V)\mapsto UV=U\cap V.$$
The lattice $L_{\open}(X)$ is complete.

The following result, which appears in~\cite{BKS}, is part of the
Stone Duality Theorem (see, for instance, \cite{Jo}).

\begin{prop}\label{pr:openlattice}
Let $X$ be a spectral space. Then $L_{\open}(X)$ is an ideal
lattice. Moreover, the map
   $$X\to\spec L_{\open}(X),\quad x \mapsto X\setminus \overline{\{x\}},$$
is a homeomorphism.
\end{prop}

We deduce from the Classification Theorem the following.

\begin{prop}\label{prr}
Let $X$ be a quasi-compact and quasi-separated scheme. Then
$\lattfl(X)$ is an ideal lattice.
\end{prop}

\begin{proof}
The space $X$ is spectral. Thus $X^*$ is spectral, also $\lato(X^*)$
is an ideal lattice by Proposition~\ref{pr:openlattice}. By the
Classification Theorem we have an isomorphism $L_{\open}(X^*)\cong
\lattfl(X)$. Therefore $\lattfl(X)$ is an ideal lattice.
\end{proof}

It follows from Proposition~\ref{rr} that $\cc S\in\lattfl(X)$ is
compact \ifff $\cc S=\cc S(Y)$ with $Y\in\lato(X^*)$ compact.

\begin{cor}\label{prrco}
Let $X$ be a quasi-compact and quasi-separated scheme. The points of
$\spec\lattfl(X)$ are the $\cap$-irreducible tensor localizing
subcategories of finite type in $\Qcoh(X)$ and the map
   $$f:X^*\to\spec\lattfl(X),\quad P\mapsto\cc S_P=\{\cc F\in\Qcoh(X)\mid
    \cc F_P=0\}$$
is a homeomorphism of spaces.
\end{cor}

\begin{proof}
This is a consequence of the Classification Theorem and
Propositions~\ref{pr:openlattice}, \ref{prr}.
\end{proof}

\section{Reconstructing quasi-compact, quasi-separated schemes}

Let $X$ be a quasi-compact and quasi-separated scheme. We shall
write $\spec(\Qcoh(X)):=(\spec\lattfl(X))^*$ and $\supp(\cc
F):=\{\cc P\in\spec(\Qcoh(X))\mid\cc F\not\in\cc P\}$ for $\cc
F\in\Qcoh(X)$. It follows from Corollary~\ref{prrco} that
   $$\supp_X(\cc F)=f^{-1}(\supp(\cc F)).$$

Following \cite{B,BKS}, we define a structure sheaf on
$\spec(\Qcoh(X))$ as follows.  For an open subset $U\subseteq
\spec(\Qcoh(X))$, let
   $$\cc S_U=\{\cc F\in\Qcoh(X)\mid\supp(\cc F)\cap U=\emptyset\}$$
and observe that $\cc S_U=\{\cc F\mid\cc F_P=0\textrm{ for all $P\in
f^{-1}(U)$}\}$ is a tensor localizing subcategory. We obtain a
presheaf of rings on $\spec(\Qcoh(X))$ by
   $$U\mapsto\End_{\Qcoh(X)/\cc S_U}(\cc O_X).$$
If $V\subseteq U$ are open subsets, then the restriction map
   $$\End_{\Qcoh(X)/\cc S_U}(\cc O_X)\to\End_{\Qcoh(X)/\cc S_V}(\cc O_X)$$
is induced by the quotient functor $\Qcoh(X)/\cc S_U\to\Qcoh(X)/\cc
S_V$. The sheafification is called the {\it structure sheaf\/} of
$\Qcoh(X)$ and is denoted by $\cc O_{\Qcoh(X)}$. Next let $\cc
P\in\spec(\Qcoh(X))$ and $P:=f^{-1}(\cc P)$. There is an affine
neighborhood $\spec R$ of $P$. We have
   $$\cc O_{\Qcoh(X),\cc P}\cong\lp_{\cc P\in V}\End_{\Rfp/\cc S_V}(R)\cong R_P\cong\cc O_{X,P}.$$
The second isomorphism follows from~\cite[\S8]{GP2}. We see that
each stalk $\cc O_{\Qcoh(X),\cc P}$ is a commutative ring. We claim
that $\cc O_{\Qcoh(X)}$ is a sheaf of commutative rings. Indeed, let
$a,b\in\cc O_{\Qcoh(X)}(U)$, where $U\in\lato(\spec(\Qcoh(X)))$. For
all $\cc P\in U$ we have $\varrho_{\cc P}^U(ab)=\varrho_{\cc
P}^U(ba)$, where $\varrho_{\cc P}^U:\cc O_{\Qcoh(X)}(U)\to\cc
O_{\Qcoh(X),\cc P}$ is the natural homomorphism. Since $\cc
O_{\Qcoh(X)}$ is a sheaf, it follows that $ab=ba$.

The next theorem says that the abelian category $\Qcoh(X)$ contains
all the necessary information to reconstruct the scheme $(X,\cc
O_X)$.

\begin{thm}[Reconstruction; cf. Rosenberg~\cite{R}]\label{coh}
Let $X$ be a quasi-compact and quasi-sepa\-rated scheme. The map of
Corollary~\ref{prrco} induces an isomorphism of ringed spaces
   $$f:(X,\cc O_X)\lra{\sim}(\spec(\Qcoh(X)),\cc O_{\Qcoh(X)}).$$
\end{thm}

\begin{proof}
The proof is similar to that of~\cite[8.3; 9.4]{BKS}. Fix an open
subset $U\subseteq\spec(\Qcoh(X))$ and consider the functor
   $$F:\Qcoh(X)\xrightarrow{(-)|_{f^{-1}(U)}}\Qcoh f^{-1}(U).$$
We claim that $F$ annihilates $\cc S_U$. In fact, $\cc F\in\cc S_U$
implies $f^{-1}(\supp(\cc F))\cap f^{-1}(U)=\emptyset$ and therefore
$\supp_X(\cc F)\cap f^{-1}(U)=\emptyset$. Thus $\cc F_P=0$ for all
$P\in f^{-1}(U)$ and therefore $F(\cc F)=0$. It follows that $F$
factors through $\Qcoh(X)/\cc S_U$ and induces a map
$\End_{\Qcoh(X)/\cc S_U}(\cc O_X)\to\cc O_X(f^{-1}(U))$ which
extends to a map $\cc O_{\Qcoh(X)}(U)\to\cc O_X(f^{-1}(U))$. This
yields the morphism of sheaves $f^\sharp\colon\cc O_{\Qcoh(X)}\to
f_*\cc O_X$.

By the above $f^\sharp$ induces an isomorphism $f^\sharp_P\colon\cc
O_{\Qcoh(X),f(P)}\to\cc O_{X,P}$ at each point $P\in X$. We conclude
that $f^\sharp_P$ is an isomorphism. It follows that $f$ is an
isomorphism of ringed spaces if the map $f:X\to\spec(\Qcoh(X))$ is a
homeomorphism. This last condition is a consequence of
Propositions~\ref{pr:openlattice}-\ref{prr} and
Corollary~\ref{prrco}.
\end{proof}

\section{Coherent schemes}

We end up the paper with introducing coherent schemes. These are
between noetherian and quasi-compact, quasi-separated schemes and
generalize commutative coherent rings. We want to obtain the
Classification and Reconstruction results for such schemes.

\begin{defs}{\rm

A scheme $X$ is {\it locally coherent\/} if it can be covered by
open affine subsets $\spec R_i$, where each $R_i$ is a coherent
ring. $X$ is {\it coherent\/} if it is locally coherent,
quasi-compact and quasi-separated.

}\end{defs}

The trivial example of a coherent scheme is $\spec R$ with $R$ a
coherent ring. There is a plenty of coherent rings. For instance,
let $R$ be a noetherian ring, and $X$ be any (possibly infinite) set
of commuting indeterminates. Then the polynomial ring $R[X]$ is
coherent. As a note of caution, however, we should point out that,
in general, the coherence of a ring $R$ does not imply that of
$R[x]$ for one variable $x$. In fact, if $R$ is a countable product
of the polynomial ring $\bb Q[y,z]$, the ring $R$ is coherent but
$R[x]$ is not coherent according to a result of Soublin~\cite{So}.
Given a finitely generated ideal $I$ of a coherent ring $R$, the
quotient ring $R/I$ is coherent.

If $R$ is a coherent ring such that the polynomial ring
$R[x_1,\ldots,x_n]$ is coherent, then the projective $n$-space $\bb
P^n_R=\Proj R[x_0,\ldots,x_n]$ over $R$ is a coherent scheme.
Indeed, $\bb P^n_R$ is quasi-compact and quasi-separated
by~\cite[5.1]{GP1} and covered by $\spec R[x_0/x_i,\ldots,x_n/x_i]$
with each $R[x_0/x_i,\ldots,x_n/x_i]$ coherent by assumption.

Below we shall need the following result.

\begin{thm}[Herzog~\cite{H}, Krause~\cite{Kr1}]\label{hk}
Let $\cc C$ be a locally coherent Grothendieck category. There is a
bijective correspondence between the Serre subcategories $\cc P$ of
$\coh\cc C$ and the localizing subcategories $\cc S$ of $\cc C$ of
finite type. This correspondence is given by the functions
   \begin{gather*}
     \cc P\longmapsto\vec{\cc P}=\{\lp C_i\mid C_i\in\cc P\}\\
     \cc S\longmapsto\cc S\cap\coh\cc C
   \end{gather*}
which are mutual inverses.
\end{thm}

\begin{prop}\label{cohsch}
Let $X$ be a quasi-compact and quasi-separated scheme. Then $X$ is a
coherent scheme if and only if $\coh(X)$ is an abelian category or,
equivalently, $\Qcoh(X)$ is a locally coherent Grothendieck
category.
\end{prop}

\begin{proof}
Suppose $X$ is a coherent scheme. We have to show that every
finitely generated subobject $\cc F$ of a finitely presented object
$\cc G$ is finitely presented. It follows from~\cite[I.6.9.10]{GD}
and Proposition~\ref{qcoh} that $\cc F\in\fg(\Qcoh(X))$ if and only
if it is locally finitely generated.

Given $P\in X$ there is an open subset $U$ of $P$ and an exact
sequence
   $$\cc O_U^n\to\cc O_U^m\to\cc G|_U\to 0.$$
By assumption, there is an affine neighbourhood $\spec R$ of $P$
with $R$ a coherent ring. Let $f\in R$ be such that $P\in
D(f)\subseteq\spec R\cap U$, where $D(f)=\{Q\in\spec R\mid f\notin
Q\}$. One has $\cc O_X(D(f))=\cc O_{R}(D(f))=R_f$, hence we get an
exact sequence
   $$\cc O_{R_f}^n\to\cc O_{R_f}^m\to\cc G|_{D(f)}\to 0.$$
Since $R$ is a coherent ring then so is $R_f$.

There is an open neighbourhood $V$ of $P$ and an epimorphism $\cc
O^k_V\twoheadrightarrow\cc F|_V$, $k\in\bb N$. Without loss of
generality, we may assume that $V=D(f)$ for some $f\in R$. It
follows that $\cc F|_{D(f)}\subset\cc G|_{D(f)}$ is a finitely
presented $\cc O_{R_f}$-module, because $R_f$ is a coherent ring.
Therefore $\cc F$ is locally finitely presented, and hence $\cc
F\in\fp(\Qcoh(X))$.

Now suppose that $\Qcoh(X)$ is a locally coherent Grothendieck
category. Given $P\in X$ and an affine neighbourhood $\spec R$ of
$P$, we want to show that $R$ is a coherent ring. The localizing
subcategory $\cc S=\{\cc F\mid\cc F|_{\spec R}=0\}$ is of finite
type, and therefore $\Qcoh(X)/\cc S$ is a locally coherent
Grothendieck category. It follows from Proposition~\ref{loc} that
$\Rfp\cong\Qcoh(\spec R)\cong\Qcoh(X)/\cc S$ is a locally coherent
Grothendieck category, whence $R$ is coherent.
\end{proof}

\begin{thm}[Classification]
Let $X$ be a coherent scheme. Then the maps
   $$V\mapsto\mathcal S=\{\cc F\in\coh(X)\mid\supp_X(\cc F)\subseteq V\}$$
and
   $$\mathcal S\mapsto V=\bigcup_{\cc F\in\mathcal S}\supp_X(\cc F)$$ induce
bijections between
\begin{enumerate}
 \item the set of all subsets of
the form $V=\bigcup_{i\in\Omega}V_i$ with quasi-compact open
complement $X\setminus V_i$ for all $i\in\Omega$,

 \item the set of all tensor Serre subcategories in $\coh(X)$.
\end{enumerate}
\end{thm}

\begin{thm}
Let $X$ be a coherent scheme. The assignments
   $$\cc T\mapsto\cc S=\{\cc F\in\coh(X)\mid\supp_X(\cc F)
     \subseteq\bigcup_{n\in\bb Z,E\in\cc T}\supp_X(H_n(E))\}$$
and
   $$\cc S\mapsto
     \{E\in\perf(X)\mid H_n(E)\in\cc S\textrm{ for all $n\in\bb Z$}\}$$
induce a bijection between
\begin{enumerate}
\item the set of all tensor thick subcategories of $\perf(X)$,

\item the set of all tensor Serre subcategories
in $\coh(X)$.
\end{enumerate}
\end{thm}

Let $X$ be a coherent scheme. The ringed space $(\spec(\coh(X)),\cc
O_{\coh(X)})$ is introduced similar to $(\spec(\Qcoh(X)),\cc
O_{\Qcoh(X)})$.

\begin{theo}[Reconstruction]
Let $X$ be a coherent scheme. Then there is a natural isomorphism of
ringed spaces
   $$f:(X,\cc O_X)\lra{\sim}(\spec(\coh(X)),\cc O_{\coh(X)}).$$
\end{theo}

The theorems are direct consequences of the corresponding theorems
for quasi-compact, quasi-separated schemes and Theorem~\ref{hk}. The
interested reader can check these without difficulty.


\begin{thebibliography}{99}

\bibitem{B1} P. Balmer, {\em Presheaves of triangulated categories and reconstruction of
             schemes}, Math. Ann. {324}(3) (2002), 557-580.
\bibitem{B} P. Balmer, {\em The spectrum of prime ideals in tensor
             triangulated categories}, J. reine angew. Math. {588} (2005), 149-168.
\bibitem{BV} A. Bondal, M. Van den Bergh, {\em Generators and representability of
             functors in commutative and noncommutative geometry}, Moscow Math. J.
             3(1) (2003), 1-36.
\bibitem{BKS} A. B. Buan, H. Krause, {\O}. Solberg, {\em Support varieties - an ideal
             approach}, Homology, Homotopy Appl. 9 (2007), 45-74.
\bibitem{E} E. Enochs, S. Estrada, {\em Relative homological algebra in the category of quasi-coherent sheaves},
            Adv. Math. 194 (2005), 284-295.
\bibitem{Ga} P. Gabriel, {\em  Des cat\'egories abeli\'ennes}, Bull. Soc. Math. France
            {90} (1962), 323-448.
\bibitem{G} G. Garkusha, {\em Grothendieck categories}, Algebra i Analiz
            13(2) (2001), 1-68. (Russian). English transl.
            in St. Petersburg Math.~J. 13(2) (2002), 149-200.
\bibitem{GP} G. Garkusha, M. Prest, {\em Classifying Serre subcategories of finitely presented
             modules}, Proc. Amer. Math. Soc., to appear.
\bibitem{GP1} G. Garkusha, M. Prest, {\em Reconstructing projective schemes from Serre subcategories},
            preprint math.AG/0608574.
\bibitem{GP2} G. Garkusha, M. Prest, {\em Torsion classes of finite type and spectra},
            Proc. ICM-2006 Satellite Conf. on K-theory and Noncomm. Geometry, to appear.
\bibitem{GD} A. Grothendieck, J. A. Dieudonn\'e, {\em El\'ements de g\'eom\'etrie
            alg\'ebrique I}, Grundlehren math. Wiss. 166, Berlin-Heidelberg-New~York, Springer-Verlag, 1971.
\bibitem{Har} R. Hartshorne, {\em Algebraic geometry}, Graduate Texts in
            Mathematics 52, Springer-Verlag, New York-Heidelberg,
            1977.
\bibitem{H} I. Herzog, {\em The Ziegler spectrum of a locally coherent Grothendieck category},
             Proc. London Math. Soc. {74}(3) (1997), 503-558.
\bibitem{Hoc} M. Hochster, {\em Prime ideal structure in commutative rings}, Trans. Amer. Math. Soc. {142} (1969), 43-60.
\bibitem{Hop} M. J. Hopkins, {\em  Global methods in homotopy theory}, Homotopy
            theory (Durham, 1985), London Math. Soc. Lecture Note
            Ser. 117, Cambridge Univ. Press, Cambridge, 1987, pp. 73-96.
\bibitem{Ho} M. Hovey, {\em Classifying subcategories of modules}, Trans. Amer. Math. Soc. {353}(8) (2001), 3181-3191.
\bibitem{Jo} P. T. Johnstone, {\em Stone Spaces}, Cambridge Studies
            in Advanced Mathematics, Vol.~3, Cambridge University Press, 1982.
\bibitem{KS} M. Kashiwara, P. Schapira, {\em Sheaves on manifolds}, Grundlehren math. Wiss. 292,
            Springer-Verlag, Berlin, 1990.
\bibitem{Kr1} H. Krause, {\em The spectrum of a locally coherent category}, J. Pure Appl.
            Algebra {114}(3) (1997), 259-271.
\bibitem{Lipman} J.~Lipman, {\em Notes on derived categories and derived
            functors}, available at www.math.purdue.edu/$\sim$lipman.
\bibitem{M} D. Murfet, {\em Modules over a scheme}, available at {\tt therisingsea.org}.
\bibitem{N} A. Neeman, {\em The chromatic tower for $D(R)$}, Topology {31}(3) (1992), 519-532.
\bibitem{Pr2} M. Prest, {\em The Zariski spectrum of the category of finitely presented
            modules}, preprint (available at {\tt maths.man.ac.uk/$\sim$mprest}).
\bibitem{R} A. L. Rosenberg, {\em The spectrum of abelian categories
            and reconstruction of schemes}, Rings, Hopf algebras, and Brauer groups, Lect. Notes Pure Appl. Math., vol. 197,
            Marcel Dekker, New York, 1998, pp. 257-274.
\bibitem{So} J.-P. Soublin, {\em Anneaux et modules coh\'erents}, J. Algebra  15 (1970), 455-472.
\bibitem{T} R. W. Thomason, {\em  The classification of triangulated subcategories},
            Compos. Math. {105}(1) (1997), 1-27.
\bibitem{Z} M. Ziegler, {\em  Model theory of modules}, Ann. Pure Appl. Logic
            {26} (1984), 149-213.

\end{thebibliography}
\end{document}